\documentclass[12pt]{article}%
\usepackage{amsfonts}
\usepackage{amsmath}
\usepackage{amssymb}
\usepackage{graphicx}%
\providecommand{\U}[1]{\protect\rule{.1in}{.1in}}
\providecommand{\U}[1]{\protect\rule{.1in}{.1in}}
\newtheorem{theorem}{Theorem}

\newtheorem{corollary}[theorem]{Corollary}

\newtheorem{definition}[theorem]{Definition}

\newtheorem{lemma}[theorem]{Lemma}

\newtheorem{proposition}[theorem]{Proposition}

\newenvironment{proof}[1][Proof]{\noindent\textbf{#1.} }{\ \rule{0.5em}{0.5em}}
\begin{document}

\title{Dependence on parameters of CW globalizations of families of Harish-Chandra
modules and the meromorphic continuation of $C^{\infty}$ Eisenstein series}
\author{Nolan R. Wallach}
\maketitle

\begin{abstract}
The first main result is that the Casselman-Wallach Globalization of a real
analytic family of Harish-Chandra modules is continuous in the parameter. Our
proof of this result uses results from the thesis of Vincent van der Noort in
several critical ways. In his thesis the holomorphic result was proved in the
case when the parameter space is a one dimensional complex manifold up to a
branched covering. The second main result is a proof of the meromorphic
continuation of $C^{\infty}$ Eisenstein series.using Langlands' results in the
$K$ finite case as an application of the methods in the proof of the first part.

\end{abstract}

\section{Introduction}

The purpose of this article is to extend my work on smooth Fr\'{e}chet
globalizations of Harish-Chandra modules to include parameters. In
\cite{BerKro} it is asserted that their work carries that goal out. This may
be so, but non-linear groups do not appear in \cite{BerKro} (e.g. the
metaplectic group). In this paper a different tactic is taken to this problem.
We approach it from the perspective of the excellent thesis of Vincent van der
Noort who studies the question: Given an analytic family of Harish-Chandra
modules, how does the corresponding family of CW completions depend on the
parameter? The CW completion was first realized in terms of imbedding into
parabolically induced representations. This paper considers another class of
Harish-Chandra modules that were first studied in a special case in
\cite{HOW}. For lack of a name they were called J--modules. These
Harish-Chandra modules are constructed using a free subalgebra of the center
of the enveloping algebra generated by the split rank number of independent
elements that was first studied in \cite{HOW}. This algebra is denoted
$\mathbf{D}$ in this paper. In the category of Harish-Chandra modules with
$\mathbf{D}$ action by a fixed character the J--modules in the category are
projective. Furthermore, every Harish-Chandra module has a resolution by
J--modules. Much of the paper, involves analyzing the CW globalizations of
families of J--modules using a key results of van der Noort, which also play
an important role in other aspects of the paper. For the sake of completeness
a complete proof of these results is included.

In van der Noort's thesis the parametrization studied were holomorphic and
results were proved about holomorphic dependence of CW globalizations. He
essentially solved the problem in the case when the parameter space is one
complex dimension modulo the possible necessity to go to a branched covering.
In this paper I prove that if the dependence of the Harish-Chandra modules in
the parameters is real analytic then the dependence of the CW completion is
continuous (Corollary \ref{main2}).

The final two sections of the paper give a proof of the meromorphic
continuation of $C^{\infty}$ Eisenstein series using the continuation of
$K$--finite Eisenstein series in Langlands \cite{FuncEisen}, Chapter 7. This
is done by reducing the problem to a general result on the holomorphic
dependence of the extensions of what we call linear functionals of locally
uniform moderate growth on families of parabolically induced Harish-Chandra
modules. Except for the use of Van der Noort's result and some notation this
part of the paper (sections 11 and 12) can be read independently of the rest
of the paper.

\section{The subalgebra \textbf{D} of $Z(\mathfrak{g)}$}

Let $G$ be a real reductive group of inner type. That is, if $\mathfrak{g}%
=Lie(G)$, $\mathfrak{g}_{\mathbb{C}}=\mathfrak{g}\otimes\mathbb{C}$ then
$Ad(G)$ is contained in the identity component of $Aut(\mathfrak{g}%
_{\mathbb{C}}).$Let $K$ be a maximal compact subgroup of $G$ and let $\theta$
denote the corresponding Cartan involution of $G$ (and of $\mathfrak{g}$). on
Set $\mathfrak{k}=Lie(K)$ and $\mathfrak{p}=\{X\in\mathfrak{g}|\theta X=-X\}$
let $p$ be the projection of $\mathfrak{g}_{\mathbb{C}}$ onto $\mathfrak{p}%
_{\mathbb{C}}$ corresponding to $\mathfrak{g}_{\mathbb{C}}=\mathfrak{k}%
_{\mathbb{C}}\oplus\mathfrak{p}_{\mathbb{C}}$. Fix a symmetric $Ad(G)$%
--invariant bilinear form, $B$, on $\mathfrak{g}$ such that $B_{|\mathfrak{k}%
}$ is negative definite and $B_{|\mathfrak{p}}$ is positive definite Extend
$p$ to a homomorphism of $S(\mathfrak{g}_{\mathbb{C}})$ onto $S(\mathfrak{p}%
_{\mathbb{C}})$. Then $p$ is the projection corresponding to%
\[
S(\mathfrak{g}_{\mathbb{C}})=S(\mathfrak{p}_{\mathbb{C}})\oplus S(\mathfrak{g}%
_{\mathbb{C}})\mathfrak{k}_{\mathbb{C}}.
\]
In \cite{HOW} we found homogeneous elements $w_{1},...,w_{l}$ of
$S(\mathfrak{g}_{\mathbb{C}})^{G}$ with $w_{1}=\sum v_{i}^{2}$ \ with
$\{v_{1},...,v_{n}\}$ and orthonormal basis of $\mathfrak{g}_{\mathbb{C}}$
with respect $B.$ Satisfying the following properties

1. $p(w_{1}),...,p(w_{l})$ are algebraically independent.

2. There exists a finite dimensional homogeneous subspace $E$ of
$S(\mathfrak{p}_{\mathbb{C}})^{K}$ such that the map $\mathbb{C[}%
p(w_{1}),...,p(w_{l})]\otimes E\rightarrow S(\mathfrak{p}_{\mathbb{C}})^{K}$
given by multiplication is an isomorphism.

If $\mathfrak{g}_{\mathbb{C}}$ contains no simple ideals of type E one can
take $E=\mathbb{C}1.$ If $\mathfrak{g}$ is split over $\mathbb{R}$ then
$\mathbb{C}[w_{1},...,w_{l}]=S(\mathfrak{g}_{\mathbb{C}})^{G}$.

Let $\mathcal{H}$ denote the space of harmonic elements of $S(\mathfrak{p}%
_{\mathbb{C}})$, that is, the orthogonal complement to the ideal
$S(\mathfrak{p}_{\mathbb{C}})\left(  S(\mathfrak{p}_{\mathbb{C}}%
)\mathfrak{p}\right)  ^{K}$ in $S(\mathfrak{p}_{\mathbb{C}})$ relative to the
Hermitian extension of inner product $B_{|\mathfrak{p}}$. Then the
Kostant-Rallis theorem (\cite{KosRal}) implies that the map%
\[
\mathcal{H}\otimes S(\mathfrak{p}_{\mathbb{C}})^{K}\rightarrow S(\mathfrak{p}%
_{\mathbb{C}})
\]
given by multiplication is a linear bijection. This and 2. easily imply

\begin{lemma}
The map
\[
\mathcal{H}\otimes E\otimes\mathbb{C}[w_{1},...,w_{l}]\otimes S(\mathfrak{k}%
_{\mathbb{C}})\rightarrow S(\mathfrak{g}_{\mathbb{C}})
\]
given by multiplication is a linear bijection.
\end{lemma}

Let $\mathfrak{a}$ be a maximal abelian subspace of $\mathfrak{p}$ and let
\[
W=\{s\in GL(\mathfrak{a})|s=Ad(k)_{|\mathfrak{a}},k\in K\}.
\]
Let $H\in\mathfrak{a}$ be such that $\mathfrak{a}=\{X\in\mathfrak{p}%
|[H,X]=0\}$. If $\lambda\in\mathbb{R}$ then set $\mathfrak{g}^{\lambda}%
=\{X\in\mathfrak{g}|[H,X]=\lambda X\}$. Set $\mathfrak{n=\oplus}_{\lambda
>0}\mathfrak{g}_{\lambda}$ and $\mathfrak{\bar{n}=\theta n=\oplus}_{\lambda
>0}\mathfrak{g}_{-\lambda}$. Then%
\[
\mathfrak{p}=p(\mathfrak{n)\oplus a}%
\]
and $p(\mathfrak{n})$ is the orthogonal complement to $\mathfrak{a}$ in
$\mathfrak{p}$ relative to $B$. Let $q$ be the projection of $\mathfrak{p}$
onto $\mathfrak{a}$ corresponding to this decomposition. Then the Chevalley
restriction theorem implies that
\[
q:S(\mathfrak{p})^{K}\rightarrow S(\mathfrak{a})^{W}%
\]
is an isomorphism of algebras. Also, as above, if $H$ is the orthogonal
complement to $\left(  S(\mathfrak{a)a}\right)  ^{W}S(\mathfrak{a)}$ in
$S(\mathfrak{a)}$. Then the map%
\[
S(\mathfrak{a})^{W}\otimes H\rightarrow S(\mathfrak{a})
\]
given by multiplication is a linear bijection. Putting these observations
together the map%
\[
S(\mathfrak{n)\otimes}S(\mathfrak{a})^{W}\otimes H\otimes S(\mathfrak{k}%
)\rightarrow S(\mathfrak{g})
\]
given by multiplication is a linear bijection. We also note that the map%
\[
\mathbb{C[}w_{1},...,w_{l}]\otimes E\rightarrow S(\mathfrak{a})^{W}%
\]
given by%
\[
w\otimes e\mapsto q(p(w))q(e)
\]
is a linear bijection. This in turn implies

\begin{lemma}
The map%
\[
S(\mathfrak{n)\otimes}\mathbb{C[}w_{1},...,w_{l}]\otimes E\otimes H\otimes
S(\mathfrak{k})\rightarrow S(\mathfrak{g})
\]
given by multiplication is a linear bijection.
\end{lemma}

Let $\mathrm{symm}$ denote the symmetrization map from $S(\mathfrak{g}%
_{\mathbb{C}})$ to $U(\mathfrak{g}_{\mathbb{C}})$ then $\mathrm{symm}$ is a
linear bijection and $\mathrm{symm}\circ Ad(g)=Ad(g)\circ\mathrm{symm}$. Let
$Z(\mathfrak{g}_{\mathbb{C}})=U(\mathfrak{g}_{\mathbb{C}})^{G}$ denote the
center of $U(\mathfrak{g}_{\mathbb{C}})$. Set $z_{i}=\mathrm{symm}(w_{i}) $
and
\[
\mathbf{D}=\mathbb{C}[z_{1},...,z_{l}].
\]
Note that if $S_{j}(\mathfrak{g}_{\mathbb{C}})=\sum_{k\leq j}S^{j}%
(\mathfrak{g}_{\mathbb{C}})$ and if $U^{j}(\mathfrak{g}_{\mathbb{C}})\subset
U^{j+1}(\mathfrak{g}_{\mathbb{C}})$ is the standard filtration of
$U(\mathfrak{g}_{\mathbb{C}})$ then
\[
\mathrm{symm}(S_{j}(\mathfrak{g}_{\mathbb{C}}))=U^{j}(\mathfrak{g}%
_{\mathbb{C}}).
\]
The above and standard arguments (\cite{HOW} Theorem 2.5 and Lemma 5.2) imply

\begin{theorem}
\label{decompositions}Let the notation be as above. Then

1. The map
\[
\mathcal{H}\otimes E\otimes\mathbf{D}\otimes U(\mathfrak{k}_{\mathbb{C}%
})\rightarrow U(\mathfrak{g}_{\mathbb{C}})
\]
given by%
\[
h\otimes e\otimes D\otimes k\mapsto\mathrm{symm}(h)\mathrm{symm}(e)Dk
\]
is a linear bijection.

2. The map%
\[
U(\mathfrak{n}_{\mathbb{C}})\otimes E\otimes H\otimes\mathbf{D\otimes
U(}\mathfrak{k}_{\mathbb{C}})\rightarrow U(\mathfrak{g}_{\mathbb{C}})
\]
given by%
\[
n\otimes e\otimes h\otimes D\otimes k\mapsto n\mathrm{symm}(e)\mathrm{symm}%
(h)Dk
\]
is a linear bijection.
\end{theorem}

\section{A class of admissible finitely generated $(\mathfrak{g},K)$--modules}

Retain the notation in the preceding section. Note that Theorem
\ref{decompositions} implies that the subalgebra $\mathbf{D}U(\mathfrak{k}%
_{\mathbb{C}})$of $U(\mathfrak{g}_{\mathbb{C}})$ is isomorphic with the tensor
product algebra $\mathbf{D}\otimes U(\mathfrak{k}_{\mathbb{C}})$ and that
$U(\mathfrak{g}_{\mathbb{C}})$ is free as a right $\mathbf{D}U(\mathfrak{k}%
_{\mathbb{C}})$ under multiplication. If $R$ is a $\mathbf{D}U(\mathfrak{k}%
_{\mathbb{C}})$--module then form
\[
J(R)=U(\mathfrak{g}_{\mathbb{C}})\otimes_{\mathbf{D}U(\mathfrak{k}%
_{\mathbb{C}})}R.
\]

Denote by $H(\mathfrak{g},K)$ the Harish--Chandra category of admissible
finitely generated $(\mathfrak{g},K)$--modules. Let $R$ be a finite
dimensional continuous $K$--module that is also a $\mathbf{D}$--module and the
actions commute then $K$ acts on $J(R)$ as follows:%
\[
k\cdot\left(  g\otimes r\right)  =Ad(k)g\otimes kr,k\in K,g\in U(\mathfrak{g}%
_{\mathbb{C}}),r\in R.
\]
Then as a $K$--module
\[
J(R)\cong\mathcal{H}\otimes E\otimes R
\]
with $K$ acting trivially on $E$. Note that $J(R)\in H(\mathfrak{g},K)$ since
the multiplicities of $K$--types in $\mathcal{H}$ are finite and $J(R)$ is
clearly finitely generated as a $U(\mathfrak{g}_{\mathbb{C}})$--module. Let
$W(\mathbf{D,}K)$ be the category of finite dimensional $(\mathbf{D}%
,K)$--modules with $K$ acting continuously and the action of $\mathbf{D}$ and
$K$ commute.

\begin{lemma}
$R\rightarrow J(R)$ defines an exact faithful functor from the category
$W(K,\mathbf{D})$ to $H(\mathfrak{g},K)$.
\end{lemma}

\begin{proof}
This follows since $U(\mathfrak{g}_{\mathbb{C}})$ is free as a module for
$\mathbf{D}U(\mathfrak{k}_{\mathbb{C}})$ under right multiplication.
\end{proof}

As usual, denote the set of equivalence classes of irreducible, finite
dimensional, continuous representations of $K$ by $\hat{K}$. If $V\in
H(\mathfrak{g},K)$ set $V(\gamma)$ equal to the sum of all irreducible
$K$--subrepresentations of $V$ in the class of $\gamma$. Then $V(\gamma)$ is
invariant under the action of $Z(\mathfrak{g}_{\mathbb{C}})$ hence under the
action of $\mathbf{D}$.

If $V\in H(\mathfrak{g},K)$ there is a finite subset $F\subset\hat{K}$ such
that
\[
U(\mathfrak{g}_{\mathbb{C}})\sum_{\gamma\in F}V(\gamma).
\]
Set $R=\sum_{\gamma\in F}V(\gamma)\in W(\mathbf{D,}K)$ and one has the
canonical $(\mathfrak{g},K)$--module surjection $J(R)\rightarrow V$ given by
$g\otimes r\mapsto gr.$ A submodule of an element of $H(\mathfrak{g},K)$ is in
$H(\mathfrak{g},K)$ so

\begin{proposition}
If $V\in H(\mathfrak{g},K)$ then there exists a sequence of elements $R_{j}\in
W(\mathfrak{g},K)$ and an exact sequence in $H(\mathfrak{g},K)\in\in$%
\[
...\rightarrow J(R_{k})\rightarrow....\rightarrow J(R_{2})\rightarrow
J(R_{1})\rightarrow J(R_{0})\rightarrow V\rightarrow0.
\]

\end{proposition}

Notice that this exact sequence us a free resolution of $V$ as a
$U(\mathfrak{n})$--module.

Let $\beta:\mathbf{D}\rightarrow\mathbb{C}$ be an algebra homomorphism. Let
$H(\mathfrak{g},K)_{\beta}$ be the full subcategory of $H(\mathfrak{g},K)$
consisting of modules $V$ such that if $z\in\mathbf{D}$ then it acts by
$\beta(z)I$. The next result is an aside that will not be used in the rest of
this paper and is a simple consequence of the definition of projective object.

\begin{lemma}
Let $F$ be a finite dimensional $K$--module and let $\mathbf{D}$ act on $F$ by
$\beta(z)I$ yielding an object $R\in W(K,\mathbf{D})$. Then $J(R)$ is
projective in $H(\mathfrak{g},K)_{\beta}$.
\end{lemma}

\section{The objects in $W(K,\mathbf{D)}$}

If $R\in W(K,\mathbf{D)}$ then $R$ has an isotypic decomposition
$R=\oplus_{\gamma\in\hat{K}}R(\gamma)$. Only a finite number of the
$R(\gamma)\neq0$. If $D\in\mathbf{D\ }$then $DR(\gamma)\subset R(\gamma)$ for
all $\gamma\in\hat{K}$. If $\chi:\mathbf{D}\rightarrow\mathbb{C}$ is an
algebra homomorphism then we set $R_{\chi}=\{v\in R|(D-\chi(D))^{k}v=0,$for
some $k>0\}$ \ Then setting $ch(\mathbf{D})$ equal to the set of all algebra
homomorphisms of $\mathbf{D}$ to $\mathbb{C}$ we have the decomposition%
\[
R=\bigoplus_{\gamma\in\hat{K},\chi\in ch(\mathbf{D})}R_{\chi}(\gamma).
\]
Fix a $K$--module $(\tau_{\gamma},F_{\gamma})\in\gamma$. Then $R_{\chi}%
(\gamma)$ is isomorphic with
\[
\mathrm{Hom}_{K}(V_{\gamma},R_{\chi})\otimes F_{\gamma}%
\]
with $K$ acting on $F_{\gamma}$ and $\mathbf{D}$ acting on $\mathrm{Hom}%
_{K}(V_{\gamma},R)$.

If $R$ is an irreducible object in $W(K,\mathbf{D)}$ then Schur's lemma
implies that $\mathbf{D}$ acts by a single homomorphism to $\mathbb{C}$ and
$R$ is irreducible as a $K$--module. Set $V_{\chi,}$ equal to the module with
$\mathbf{D}$ acting by $\chi$ and $K$ acting by an element of $\gamma$.

We next analyze the homomorphisms $\chi$. Let $\chi$ be such a homomorphism
then $\chi(z_{i})=\lambda_{i}\in\mathbb{C}$. Thus one simple parametrization
is by $(\lambda_{1}....,\lambda_{l})\in\mathbb{C}^{l}$. We use the notation
$\beta_{\lambda}$ for the homomorphism such that $\beta_{\lambda}%
(z_{i})=\lambda_{i}$. An alternate parametrization is through the
Harish-Chandra homomorphism. Recall the exact sequence (c.f. \cite{RRG1-11},
Theorem 3.6.6)%
\[%
\begin{array}
[c]{ccccccccc}
&  &  &  &  & \mathbf{\gamma} &  &  & \\
0 & \rightarrow & (U(\mathfrak{g}_{\mathbb{C}})\mathfrak{k}_{\mathbb{C}})^{K}
& \rightarrow & U(\mathfrak{g}_{\mathbb{C}})^{K} & \rightarrow &
U(\mathfrak{a}_{\mathbb{C}})^{W} & \rightarrow & 0.
\end{array}
\]
It is standard that the linear map $\mathbf{\gamma}\circ\mathrm{symm}%
:S(\mathfrak{p}_{\mathbb{C}})^{K}\rightarrow U(\mathfrak{a}_{\mathbb{C}})^{W}$
is a linear bijection. This and the definition of $\mathbf{D}$ imply that
$U(\mathfrak{a}_{\mathbb{C}})^{W}$ is finitely generated as a $\mathbf{\gamma
}(\mathbf{D})$--module. \ This in turn implies that $U(\mathfrak{a}%
_{\mathbb{C}})$ is finitely generated as a $\mathbf{\gamma}(\mathbf{D}%
)$--module. Thus we have a morphism $\varphi:\mathfrak{a}_{\mathbb{C}}^{\ast
}\rightarrow\mathbb{C}^{l}$ such that $\mathbf{\gamma}(\mathbf{D})(\nu)$ is
the homomorphism $z_{i}\mapsto\varphi_{i}(\nu)$. Hence $\mathbf{\gamma
}(\mathbf{D})(\nu)=\beta_{\varphi(\nu)}$ for $\nu\in\mathfrak{a}_{\mathbb{C}%
}^{\ast}$. Set $\chi_{\nu|_{\mathbf{D}}}=\beta_{\varphi(\nu)} $.

\begin{definition}
Let $X$ be a complex or real analytic manifold. An analytic family in
$W(K,\mathbf{D)}$ based on $X$ is a pair $(\mu,V)$ of a a finite dimensional
continuous $K$--module,$V$, and a $\mu:X\times\mathbf{D}\rightarrow
\mathrm{End}(V)$ such that $D\mapsto\mu(x,D)$ is a representation of
$\mathbf{D}$ on $V$ and $x\mapsto\mu(x,D)$ is analytic for all $D\in
\mathbf{D}$.
\end{definition}

\section{Analytic families of $J$--modules}

Throughout this section analytic will mean complex analytic in the context of
a complex analytic manifold and real analytic in the contest of a real
analytic manifold. Theorem \ref{decompositions} implies

\begin{corollary}
\label{covar}Let $R\in W(K,\mathbf{D})$ then
\[
J(R)/\mathfrak{n}^{k+1}J(R)\cong U(\mathfrak{n)/}\mathfrak{n}^{k+1}%
U(\mathfrak{n})\otimes E\otimes H\otimes R_{|M}%
\]
as an $(\mathfrak{n},M)$-module with $\mathfrak{n}$ and $M$ acting trivially
on $E\otimes H$ and $\mathfrak{n}$ acting trivially on $R$.
\end{corollary}

Let $(\mu,V)$ be an analytic family of objects in $W(K,\mathbf{D)}$ based on
$X$. Let $V_{x},x\in X$ be the object in $W(K,\mathbf{D)}$ with $K$ acting by
its action on $V$ and $\mathbf{D}$ action by $\mu_{x}=\mu(x,\cdot)$.$.$

\begin{theorem}
\label{dependence}Notation as above. Let $\sigma_{k,x}$ be the action of
$\mathfrak{a}$ on
\[
U(\mathfrak{n)/}\mathfrak{n}^{k+1}U(\mathfrak{n})\otimes E\otimes H\otimes
V_{x|M}%
\]
under the identification
\[
J(V_{x})/\mathfrak{n}^{k+1}J(V_{x})\cong U(\mathfrak{n)/}\mathfrak{n}%
^{k+1}U(\mathfrak{n})\otimes E\otimes H\otimes V_{x|M}.
\]
If $u\in U(\mathfrak{g}_{\mathbb{C}})$ then the map $x\rightarrow\sigma
_{k,x}(u)$ is an analytic map.
\end{theorem}

\begin{proof}
Theorem \ref{decompositions} implies that if $X_{1},...,X_{m}$ is a basis of
$\mathfrak{n}$, $Y_{1},...,Y_{n}$ is a basis of $\mathfrak{k}$ and
$h_{1},...,h_{r}$ is a basis for $\mathrm{symm}(E)\mathrm{symm}(H)$ then if
$I,J,L$ are multi-indices of size $m,n,l$ respectively then%
\[
X^{J}z^{L}h_{i}Y^{L}%
\]
is a basis of $U(\mathfrak{g}_{\mathbb{C}})$. Here, as usual, $X^{J}%
=X_{1}^{j_{1}}\cdots X_{m}^{j_{m}}$, ... This implies that if $u\in
U(\mathfrak{g}_{\mathbb{C}})$ then
\[
uX^{J}z^{L}h_{i}Y^{J}=\sum a_{I_{1}L_{1}i_{1}J_{1},ILiJ}(u)X^{J_{1}}z^{L_{1}%
}h_{i_{1}}Y^{J_{1}}.
\]
This implies that if we take a basis $v_{1},...,v_{d}$ of $V$ then the
elements $X^{J}h_{i}\otimes v_{j}$ form a basis of $J(V_{x})$. Thus if $u\in
U(\mathfrak{g}_{\mathbb{C}})$ then
\[
uX^{J}h_{i}\otimes v_{j}=\sum a_{I_{1}L_{1}i_{1}J_{1},I,0i0}(u)X^{J_{1}%
}z^{L_{1}}h_{i_{1}}Y^{J_{1}}\otimes v_{j}=
\]%
\[
\sum a_{I_{1}L_{1}i_{1}J_{1},I0i0}(u)X^{J_{1}}h_{i_{1}}\otimes\mu_{x}%
(z^{L_{1}})Y^{J_{1}}v_{j}=
\]
The theorem follows from this formula.
\end{proof}

If $X$ is a complex manifold or a real analytic manifold and $V$ is a vector
space over $\mathbb{C}$ then a map $\phi:X\rightarrow V$ is said to be
holomorphic, real analytic or continuous if for each $x\in X$ there exists a
open neighborhood, $U$, \ of $X$ such that if $Z=\mathrm{Span}_{\mathbb{C}%
}\{\phi(x)|x\in U\}$ then $\dim Z<\infty$ and $\phi:U\rightarrow Z$ is
holomorphic,real analytic or continuous respectively.

\begin{definition}
Let $X$ be a complex or real analytic manifold. Then an holomorphic,analytic
or continuous family of admissible $(\mathfrak{g},K)$--modules based on $X$ is
a pair, $(\mu,V)$, of an admissible $(\mathfrak{k},K)$--module, $V$, and
\[
\mu:X\times U(\mathfrak{g})\rightarrow\mathrm{End}(V)
\]
such that $x\mapsto\mu(x,y)v$ is holomorphic (resp. analytic, resp.
continuous) for all $y\in U(\mathfrak{g})$, $v\in V$ and if we set $\mu
_{x}(y)=\mu(x,y)$ for $y\in U(\mathfrak{g})$ then $(\mu_{x},V)$ is an
admissible finitely $(\mathfrak{g},K)$--module. It will be called a family of
objects in $H\mathcal{(}\mathfrak{g},K)$ if each $(\mu_{x},V)$ is finitely generated.
\end{definition}

\begin{theorem}
\label{J-family}Let $X$ be an analytic or complex manifold. Let $(\lambda,R) $
be an family of objects in $W(K,\mathbf{D})$ based on $X$ and define $R_{x}\in
W(K,\mathbf{D})$ to be the module with action $\lambda(x,\cdot)$. Let
\[
V=\mathcal{H}\otimes E\otimes R
\]
$(K$ act by the tensor product action with its action on $E$ trivial) and let
$T_{x}:V\rightarrow J(R_{x})$ be given by $T_{x}(h\otimes e\otimes
r)=\alpha_{x}(\mathrm{symm}(h)e)(1\otimes r)$\ with $\alpha_{x}$ the action of
$U(U(\mathfrak{g}_{\mathbb{C}})$ on $J(R_{x})$. If $\lambda(x,y)=T_{x}%
^{-1}\alpha_{x}(y)T_{x}$ then $(\lambda,V)$ is an analytic family of objects
in $H(\mathfrak{g},K)$ based on $X$.
\end{theorem}

\begin{proof}
We argue as in the proof of Theorem \ref{dependence}. Let $\left\{
h_{i}\right\}  $ be a basis of $\mathcal{H}$ such that for each $i$ there
exists $\gamma\in\hat{K}$ such that $h_{i}\in\mathcal{H(\gamma)}$, let $e_{j}
$ be a basis of $E$, let $r_{m}$ be a basis of $R$ and let $Y_{1},...,Y_{n}$
be a basis of $\mathfrak{k}$. Then if $y\in U(\mathfrak{g}_{\mathbb{C}})$
\[
y\mathrm{symm}(h_{i})e_{j}z^{L}Y^{J}=\sum_{i_{1},j_{1},J_{1},L_{1}}%
b_{i_{1}j_{1}L_{1}J_{1},ijLK}(y)\mathrm{symm}(h_{i_{1}})e_{j_{1}}z^{J_{1}%
}Y^{K_{1}}.
\]
Thus%
\[
T_{x}^{-1}\alpha_{x}(y)T_{x}(h_{i}\otimes e_{j}\otimes r_{k})=\sum
b_{i_{1}j_{1}L_{1}J_{1},ij00}(y)h_{i_{1}}\otimes e_{j_{1}}\otimes\lambda
_{x}(z^{L_{1}})Y^{J_{1}}r_{k}.
\]
The theorem follows.
\end{proof}

Next we define another type of analytic family. Let $A$ and $N$ be the
connected subgroups of $G$ with $Lie(A)=\mathfrak{a}$ and $Lie(N)=\mathfrak{n}%
$. Let $M$ be the centralizer of $\mathfrak{a}$ in $K$. Set $Q=MAN$ then $Q$
is a minimal parabolic subgroup of $G$.

\begin{definition}
An analytic family of finite dimensional $Q$--modules based on the manifold
(real or complex analytic) $X$ is a pair $(\sigma,S)$ with $S$ a finite
dimensional continuous $M$--module and a real analytic map $\sigma:X\times
Q\rightarrow GL(S)$ such that $x\mapsto\sigma(x,q)$ is holomorphic and
$\sigma(x,\cdot)=\sigma_{x}$ is a representation of $Q$.
\end{definition}

Let $(\sigma,S)$ be a continuous representation of $Q$. Set $I^{\infty}%
(\sigma_{|M})$ equal to the space of all smooth $f:K\rightarrow S$ satisfying
$f(mk)=\sigma(m)f(k)$. Define and action $\pi_{\sigma}$ of $G$ on $I^{\infty
}(\sigma_{|M})$ as follows: if $f\in I^{\infty}(\sigma_{|M}) $ then extend $f$
to $G$ by $f_{\sigma}(qk)=\sigma(q)f(k)$, then, since $K\cap Q=M$ and $QK=G$,
$f_{\sigma}$ is $C^{\infty}$ on $G$ set $\pi_{\sigma}(g)f(k)=f_{\sigma}(kg)$.
Also set%
\[
\pi_{\sigma}(Y)f(k)=\frac{d}{dt}f_{\sigma}(k\exp tY)_{|t=0}%
\]
for $Y\in\mathfrak{g}$ and $k\in K,f\in I^{\infty}(\sigma_{|M})$. Let
$I(\sigma_{M})$ be the space of all right $K$ finite elements of $I^{\infty
}(\sigma_{|M})$

Put and $M$--invariant inner product, $\left\langle ...,...\right\rangle $ on
$S$. If $f,h\in I^{\infty}(\sigma_{|M})$ then set
\[
(f,h)=\int_{K}\left\langle f(k),h(k)\right\rangle dk
\]
with $dk$ normalized invariant measure on $K$. The following is standard.

\begin{proposition}
Let $(\sigma,S)$ be an analytic family of finite dimensional representations
of $Q$ based on the complex or real analytic manifold $X$. Set $\lambda
(x,y)=\pi_{\sigma_{x}}(y)$ for $x\in X,y\in U(\mathfrak{g}_{\mathbb{C}})$. If
$\mu$ is the common value of $\sigma_{x}|_{M}$, then $(\lambda,I(\mu))$ is an
analytic family of objects in $H(\mathfrak{g},K)$ based on $X$.
\end{proposition}

\begin{proof}
It is standard that
\[
x,g\mapsto(\pi_{\sigma_{x}}(g)f,h)
\]
is real analytic and holomorphic in $x$ for $f,g\in I(\mu)$.
\end{proof}

\begin{definition}
If $(\lambda,V)$ and $(\mu,W)$ are analytic families of objects in
$H(\mathfrak{g},K)$ based on the manifold $X$ then a homomorphism of the
analytic (resp. continuous) family $(\lambda,V)$ to $(\mu,W)$ is a map
\[
T:X\rightarrow\mathrm{Hom}_{\mathbb{C}}(V,W)
\]
such that

1. $x\mapsto T(x)v$ is an analytic map of $X$ to $W$ for all $v\in V.$

2. $T(x)\in\mathrm{Hom}_{H(\mathfrak{g},K)}(V_{x},W_{x})$ with $V_{x}%
=(\lambda_{x},V),W_{x}=(\mu_{x},W)$.
\end{definition}

If $R\in W(K,\mathbf{D})$ the space $J(R)/\mathfrak{n}^{s+1}J(R)$ has a
natural structure of an $M$ module and an $\mathfrak{n}+\mathfrak{a}$ module.
Since $\dim J(R)/\mathfrak{n}^{s+1}J(R)<\infty$ and $AN$ is a simply connected
Lie group $J(R)/\mathfrak{n}^{s+1}J(R)$ has a natural structure of a finite
dimensional continuous $Q$--module with action $\sigma_{s,R}$. Let $p_{s}$
denote the natural surjection%
\[
p_{s}:J(R)\rightarrow J(R)/\mathfrak{n}^{s+1}J(R).
\]
If $k\in K$, $v\in J(R)$, define $S_{s,R}(v)(k)=p_{s,R}(kv)$, then
$S_{s,R}(v)\in I(\sigma_{s,R}|_{M})$ and it is easily seen that $S_{s,R}%
\in\mathrm{Hom}_{H(\mathfrak{g},K)}(J(R),(\pi_{\sigma_{s,R}},I(\sigma
_{s,R}|_{M}))$. Combining the above results we have

\begin{theorem}
\label{J-morph}Let $(\mu,R)$ be an analytic (resp. continuous) family in
$W(K,\mathbf{D})$ based on the manifold $X$. Let $(\lambda,V)$ be the analytic
family (as in Theorem \ref{J-family}) corresponding to $x\rightarrow
J((\mu_{x},R))$. Then recalling that $V=\mathcal{H\otimes}E\otimes R$ define
$T_{s}(x)(h\otimes e\otimes r)=S_{s},_{R_{x}}(\mathrm{symm}(h)e\otimes r).$
Then $T_{s}$ defines a homomorphism of the analytic family $(\lambda,V)$ to
$(\xi,I(\sigma_{s,R_{x},}|_{M}))$ \ with $\xi(x,y)=\pi_{\sigma_{s,R_{x}}}(y)$
and $\sigma_{s,R_{x}}$ is defined as in Theorem \ref{dependence}.
\end{theorem}

We will use the notation $J(R)$ for the analytic family associated with
$x\rightarrow J((\mu_{x},R))$.

\section{Some results of Vincent van der Noort}

Throughout this section $Z$ will denote a connected real or complex analytic
manifold. We will use the terminology analytic to mean complex analytic or
real analytic depending on the context.

We continue the notation of the previous sections. In particular $G$ is a real
reductive group of inner type.

We denote (as is usual) the standard filtration of $U(\mathfrak{g)}$, by
\[
...\subset U^{j}(\mathfrak{g)}\subset U^{j+1}(\mathfrak{g)\subset...}%
\]
Let $V$ be an admissible $(Lie(K),K)$ module. We note that if $E\subset V$ is
a finite dimensional $K$--invariant subspace of $V$ then there exists a finite
subset $F_{j,E}\subset\hat{K}$ such that
\[
U^{j}(\mathfrak{g})\otimes E\cong\sum_{\gamma\in F_{j,E}}m_{\gamma,j}%
V_{\gamma}.
\]
If $v\in V$ we denote by $E_{v}$ the span of $Kv$ in $V$.

The purpose of this section is to prove a theorem of van der Noort which first
appeared in his thesis \cite{VanderNoort}. Our argument follows his original
line with a few simplifications. We include the details only because he is not
expected to publish it. In his thesis he emphasized the holomorphic case.

Fix a maximal torus, $T$, of $M$ then $Lie(T)\oplus\mathfrak{a}$ is a Cartan
subalgebra of $\mathfrak{g}$. Set $\mathfrak{h}$ equal to its
complexification. We parametrize the homomorphisms of $Z(\mathfrak{g})$ to
$\mathbb{C}$ by $\chi_{\Lambda}$ for $\Lambda\in\mathfrak{h}^{\ast}$using the
Harish--Chandra parametrization. Endow $\hat{M}$ with the discrete topology.
Then we note that if $C$ is a compact subset of $\hat{M}\times\mathfrak{a}%
_{\mathbb{C}}^{\ast}$ then there exist a finite number of elements $\xi
_{1},...,\xi_{r}\in\hat{M}$ and compact subsets , $D_{j}$, of $\mathfrak{a}%
_{\mathbb{C}}^{\ast}$ such that%
\[
C=\cup_{j=1}^{r}\xi_{j}\times D_{j}\text{.}%
\]

If $\xi\in\hat{M}$ and $\nu\in\mathfrak{a}_{\mathbb{C}}^{\ast}$ then set
$\sigma_{\xi,\nu}(man)=\xi(m)a^{\nu+\rho}$ ($\rho(H)=\frac{1}{2}%
tr(adH_{|Lie(N)})$)$,$ $H\in\mathfrak{a}$), $a^{\nu}=\exp(\nu(H))$ $a=\exp
(H)$, $\xi$ is taken to be a representative of the class $\xi$. $H^{\xi,\nu}$
is $I(\sigma_{\xi.\nu})$ which equals as a $K$--module $H^{\xi}=Ind_{M}%
^{K}(\xi$\thinspace$)$. If $f\in H^{\xi}$ set $f_{\nu}(nak)=a^{\nu+\rho
}f(k),n\in N,a\in A,k\in K$. $A_{\bar{P}}(\nu)$ is the corresponding
Kunze-Stein intertwining operator (c.f. \cite{HarHomSP}, 8.10.18. p.241).

\begin{proposition}
\label{indfinite}Let $\xi\in\widehat{M}$ and let $\Omega\subset\mathfrak{a}%
_{\mathbb{C}}^{\ast}$ be open with compact closure. Then there exists
$F\subset\widehat{K}$ such that $\pi_{\xi,\nu}(U(\mathfrak{g}))\left(
\sum_{\gamma\in F}H^{\xi}(\gamma)\right)  =H^{\xi}$ for all $\nu\in\Omega$.
\end{proposition}

The proof of this result will use the following lemma.

\begin{lemma}
If $\nu_{o}\in\mathfrak{a}_{\mathbb{C}}^{\ast}$ then there exists an open
neighborhood of $\nu_{o}$, $U_{\nu_{o}}$, and a finite subset $F=F_{\nu_{o}}$
of $\widehat{K}$ such that $\pi_{\xi,\nu}(U(\mathfrak{g}))\left(  \sum
_{\gamma\in F}H^{\xi}(\gamma)\right)  =H^{\xi}$ for all $\nu\in U_{\nu_{o}}$.
\end{lemma}

\begin{proof}
If $\gamma\in\hat{K}$ fix $W_{\gamma}\in\gamma$. If $\operatorname{Re}%
(\nu,\alpha)>0$ for all $\alpha\in\Phi^{+}$ and if $\gamma\in\widehat{K}$ and
$A_{\overline{P}}(\nu)H^{\xi}(\gamma)\neq0$ then $\pi_{\xi,\nu}(U(\mathfrak{g}%
))\left(  H^{\xi}(\gamma)\right)  =H^{\xi}$(c.f. \cite{RRG1-11}, Theorem 5.4.1
(1)). Fix such a $\gamma_{\nu}$ (which always exists since the operator
$A_{\overline{P}}(\nu)\neq0$), take $F_{\nu}=\{\gamma_{\nu}\}$ and $U_{\nu}$
an open neighborhood of $\nu$ such that $A_{\overline{P}}(\mu)H^{\xi}%
(\gamma_{\nu})\neq0$ for $\mu\in U$. Let $\nu\in\mathfrak{a}_{\mathbb{C}%
}^{\ast}$ be arbitrary. There exists a positive integer, $k$, such that
$\operatorname{Re}(\nu+k\rho,\alpha)>0$ for all $\alpha\in\Phi^{+}$ and such
that $k\rho$ is the highest weight of a finite dimensional spherical
representation, $V^{k\rho},$ of $G$ relative to $\mathfrak{a}$. The lowest
weight of $V^{k\rho}$ relative to $\mathfrak{a}$ is $-k\rho$ and $M$ acts
trivially on that weight space thus $H_{K}^{\xi,\nu+k\rho}\otimes V^{k\rho}$
has $H_{K}^{\xi,\nu}$ as a quotient representation (see \cite{HarHomSP}%
,8.5.14,15). Take $F_{\nu}$ to be the set of $K$--types that occur in both
$W_{\gamma_{\nu+k\rho}}\otimes V^{k\rho}$ and $H^{\xi}$ and $U_{\nu}%
=U_{\nu+k\rho}-k\rho$.
\end{proof}

We now prove the proposition. By the lemma above for each $\nu\in
\overline{\Omega}$ there exists $F_{\nu}$ and $U_{\nu}$ as in the statement of
the lemma. The $U_{\nu}$ form an open covering of $\overline{\Omega}$ which is
assumed to be compact. Thus there exist a finite number $\nu_{1},...,\nu
_{r}\in\overline{\Omega}$ such that
\[
\overline{\Omega}\subset\cup_{i=1}^{r}U_{\nu_{i}}\text{.}%
\]
Take $F=\cup_{i-1}^{r}F_{\nu_{i}}.$ This proves the proposition.

\begin{lemma}
Let $\chi_{\xi},_{\nu}$ denote the infinitesimal character of $\pi_{\xi,\nu}$.
If $C$ is a compact subset of $\mathfrak{h}_{K}^{\ast}$ then
\[
\{(\xi,\nu)\in\{\hat{M}\times\mathfrak{a}_{\mathbb{C}}^{\ast}|\chi_{\xi}%
,_{\nu}=\chi_{\Lambda},\Lambda\in C\}
\]
is compact.
\end{lemma}

\begin{proof}
Fix a system of positive roots for $(M^{0},T)$ ($M^{0}$ the identity component
of $M$). If $\lambda_{\xi}$ is the highest weight of $\xi$ relative to this
system of positive roots and if $\rho_{M}$ is the half sum of these positive
roots then $\chi_{\xi},_{\nu}=\chi_{\Lambda}$ with $\Lambda=\lambda_{\xi}%
+\rho_{M}+\nu$. This implies the lemma.
\end{proof}

\begin{lemma}
\label{VanNort-Finite}Let $(\pi,V)$ be an analytic family of admissible
$(\mathfrak{g},K)$ modules based on $Z$. Assume that $z_{0}\in Z$ is such that
$(\pi_{z_{o}},V)$ is finitely generated. If $T$ is an element of
$Z(\mathfrak{g})$ there exist analytic functions $a_{0},...,a_{n-1}$ on $Z$
such that if $z\in Z$ and $\mu$ is an eigenvalue of $\pi_{z}(T)$ then $\mu$ is
a root in $x$ of
\[
f(z,x)=x^{n}+\sum_{j=0}^{n-1}a_{j}(z)x^{j}.
\]

\end{lemma}

\begin{proof}
Let $F$ be a finite number of elements of $\hat{K}$ such that $\pi_{z_{0}%
}(U(\mathfrak{g}))\sum_{\gamma\in F}V(\gamma)=V$. Let $L=$ $\sum_{\gamma\in
F}V(\gamma)$. Then we define the $a_{j}$ the by the formula
\[
f(z,x)=\det\left(  xI-\pi_{z}(T\right)  _{|L})=x^{n}+\sum_{j=0}^{n-1}%
a_{j}(z)x^{j}.
\]
The Cayley-Hamilton theorem implies that $h(z)=T^{n}+\sum_{j=0}^{n-1}%
a_{j}(z)T^{j}\in Z(\mathfrak{g})$ vanishes on $L$. Let $\gamma\in\hat{K}$ then
there exist $x_{1},...,x_{r}\in U(\mathfrak{g})$ and $v_{1},...,v_{r}\in L$
such that $\{\pi_{z_{0}}(x_{i})v_{i}\}_{i=1}^{r}$ is a basis of $V(\gamma)$.
Let $P_{\gamma}$ be the projection onto the $\gamma$--isotypic component of
$V$. Thus%
\[
(P_{\gamma}\pi_{z}(x_{1})v_{1})\wedge(P_{\gamma}\pi_{z}(x_{2})v_{2}%
)\wedge\cdots\wedge(P_{\gamma}\pi_{z}(x_{r})v_{r})\in\wedge^{r}V(\gamma)
\]
(a one dimensional space) is non--zero for $z=z_{0}$. This implies that there
exists an open neighborhood, $U$, of $z_{0}$ in $\Omega$ such that
\[
P_{\gamma}\pi_{z}(x_{1})v_{1},P_{\gamma}\pi_{z}(x_{2})v_{2},...,P_{\gamma}%
\pi_{z}(x_{r})v_{r}%
\]
is a basis of $V(\gamma)$ for $z\in U$. That
\[
h(z)P_{\gamma}\pi_{z}(x_{i})v_{i}=P_{\gamma}\pi_{z}(x_{i})h(z)v_{i}=0
\]
implies that $h(z)V(\gamma)=0$ for $z\in U$. The connectedness of $Z$ implies
that $h(z)V(\gamma)=0$ for $z\in Z$. Thus $h(z)=0$ for all $z\in Z$. This
proves the Lemma.
\end{proof}

If $V$ is a $(\mathfrak{g},K)$--module then set $ch(V)$ equal to the set of
$\Lambda\in\mathfrak{h}^{\ast}$ such that there exists $v\in V$ with
$Tv=\chi_{\Lambda}(T)v$ for all $T\in Z(\mathfrak{g})$.

\begin{corollary}
\label{compactness}Keep the notation and assumptions of the previous lemma, If
$\omega\subset Z$ is compact then there exists a compact subset $C$ of
$\mathfrak{h}^{\ast}$ such that $ch(\pi_{z},V)\subset C$ for all $z\in\omega$.
\end{corollary}

\begin{proof}
Let $T_{1},...,T_{m}$ be a generating set for $Z(\mathfrak{g})$ and let
$f_{j}(z,x)$ be the function in the previous lemma corresponding to $T_{j}$.
Then
\[
f_{j}(z,x)=x^{n_{j}}+\sum_{i=0}^{n_{j}-1}a_{j,i}(z)x^{j}%
\]
with $a_{j,i}$ analytic in on $Z$. If $\chi_{\Lambda}\in ch(\pi_{z},V)$ then%
\[
\left\vert \chi_{\Lambda}(Z_{j})\right\vert \leq\max_{0\leq i<n_{j}}%
|a_{j,i}(z)|+1
\]
(c.f. \cite{RRG1-11},7.A.1.3). If $C\subset Z$ is compact then there exists a
constant $r<\infty$ such that $|a_{j,i}(z)|\leq r$ for all $i,j$ and $z\in C$.
This implies the corollary.
\end{proof}

\begin{theorem}
\label{loc-finite}Let $(\pi,V)$ be an analytic family of admissible
$(\mathfrak{g},K)$ modules based on $Z$. Assume that there exists $z_{0}\in Z$
such that $(\pi_{z_{0}},V)$ is finitely generated. If $\omega$ is a compact
subset of $Z$ then there exists $S_{\omega}\subset\hat{K}$ a finite subset
such that if $y\in\omega$ then%
\[
\pi_{y}(U(\mathfrak{g))}\left(  \sum_{\gamma\in S_{\omega}}V(\gamma)\right)
=V\text{.}%
\]

\end{theorem}

\begin{proof}
Let $C$ as in the above corollary for $\omega$. Let
\[
X=\{(\xi,\nu)\in\hat{M}\times\mathfrak{a}_{\mathbb{C}}^{\ast}|\chi_{\xi}%
,_{\nu}=\chi_{\Lambda},\Lambda\in C\}.
\]
$X$ is compact so there exist $\xi_{1},...,\xi_{r}\in\hat{M}$ and
$D_{1},...,D_{r}$, compact subsets of $\mathfrak{a}_{\mathbb{C}}^{\ast}$, such
that $X=\cup_{j}\xi_{j}\times D_{j}$. Let $S_{j}\subset\hat{K}$ be the finite
set corresponding to $\xi_{j}\times D_{j}$ in Proposition \ref{indfinite}. Set
$S_{\omega}=\cup S_{j}$. Let $L_{1}\subset L_{2}\subset...\subset L_{j}%
\subset...$ be an exhaustion of the $K$--types of $V$ with each $L_{j}$ finite.

We will use the notation \thinspace$V_{y}$ for the $(\mathfrak{g},K)$--module
$(\pi_{y},V)$. Let $y\in C$. Set $W_{j}=\pi_{y}(U(\mathfrak{g))}\left(
\sum_{\gamma\in L_{j}}V(\gamma)\right)  $ then $W_{j}\subset W_{j+1}$ and
$\cup W_{j}=V$. Each $W_{j}$ is finitely generated and admissible, hence of
finite length. Therefore $V_{y}$ has a finite composition series
\[
0=V_{y}^{0}\subset V_{y}^{1}\subset...\subset V_{y}^{N}%
\]
or a countably infinite composition series%
\[
0=V_{y}^{0}\subset V_{y}^{1}\subset...\subset V_{y}^{n}\subset V_{y}%
^{n+1}\subset...
\]
with $V_{y}^{i}/V_{y}^{i-1}$ irreducible. Thus by the dual form of the
subrepresentation theorem there exists for each $i,\xi_{i}\in\hat{M}$ and
$\nu_{i}\in\mathfrak{a}_{\mathbb{C}}^{\ast}$ so that $V_{y}^{i}/V_{y}^{i}$ is
a quotient of $(\pi_{\xi_{i},\nu_{i}},H^{\xi_{i},\nu_{i}})$. Observe that
$(\xi_{i},\nu_{i})\in X$. Thus $V_{y}^{i}/V_{y}^{i-1}(\gamma_{i})\neq0$ for
some $\gamma_{i}\in S_{\omega}$. Let $M$ be a quotient module of $V_{y}$. Then
$M=V_{y}/U$ with $U$ a submodule of $V_{y}$. There must be an $i$ such that
$V_{y}^{i}/\left(  V_{y}^{i-1}\cap U\right)  \neq0$. Let $i$ be minimal
subject to this condition. Then $V_{y}^{i-1}\subset U$. Thus $V_{y}^{i}%
/V_{y}^{i-1}$ is a submodule of $M$. Hence $M(\gamma)\neq0$ for some
$\gamma\in S_{\omega}$. This implies that
\[
\pi_{y}(U(\mathfrak{g))}\left(  \sum_{\gamma\in S_{\omega}}V(\gamma)\right)
=V\text{.}%
\]
Indeed,
\[
\left(  V_{y}/\pi_{y}(U(\mathfrak{g))}\left(  \sum_{\gamma\in S_{\omega}%
}V(\gamma)\right)  \right)  (\gamma)=0,\gamma\in S_{\omega}.
\]

\end{proof}

\begin{corollary}
\label{finitegen}(To the proof) Let $(\pi,V)$ be an analytic family of
finitely generated admissible $(\mathfrak{g},K)$ modules based on $Z$ . Let
$\omega$ be open in $Z$ with compact closure. Let for each $z\in\omega$,
$U_{z}$ be a $(\mathfrak{g},K)$--submodule of $V_{z}$. Then there exists a
finite subset $F_{\omega}\subset\hat{K}$ such that
\[
\pi_{z}(U(\mathfrak{g}))\left(  \sum_{\gamma\in F_{\omega}}U_{z}%
(\gamma)\right)  =U_{z}.
\]

\end{corollary}

\begin{proof}
In the proof of the theorem above all that was used was that the set of
possible infinitesimal characters is compact.
\end{proof}

\section{Imbeddings of families of $J$--modules}

Let $X$ be a connected real or complex analytic manifold and let $(\mu,R)$ be
an analytic family of objects in $W(K,\mathbf{D})$ based on $X.$ The purpose
of this section is to prove

\begin{theorem}
\label{Imbedding}Let the representation of $Q$, $\sigma_{k,x}$, on
\[
W_{k}=U(\mathfrak{n)/}\mathfrak{n}^{k+1}U(\mathfrak{n})\otimes E\otimes
H\otimes R_{|M}%
\]
be as in Theorem \ref{dependence} and let $T_{k}(x)$ be the analytic family as
in Theorem \ref{J-morph}. If $\omega$ is a compact subset of $X$ then there
exists $k_{\omega}$ such that if $x\in\omega$ then $T_{k}(x)$ is injective.
\end{theorem}

This is a slight extension of a result in \cite{HOW}. Given $k$ then
$(\sigma_{k,x},W_{k})$ as a composition series $W_{k,x}=W_{k,x}^{1}\supset
W_{k,x}^{2}\supset...\supset W_{k,x}^{r}\supset W_{k,x}^{r+1}=\{0\}$ and each
$W_{k,x}^{i}/W_{k,x}^{i+1}$ is isomorphic with the representation
$(\lambda_{j,\nu_{j}},H_{\lambda_{j}})$ with $(\lambda_{j},H_{j})$ an
irreducible representation of $M$ and $\nu_{j}\in\mathfrak{a}_{\mathbb{C}%
}^{\ast}$ and $\lambda_{j,\nu}(man)=a^{\nu+\rho}\lambda_{j}(m)$ with $m\in
M,a\in A$ and $n\in N$. Also note that there is a natural $Q$--module exact
sequence%
\[
0\rightarrow\mathfrak{n}^{k+2}U(\mathfrak{n)/n}^{k+1}U(\mathfrak{n})\otimes
E\otimes H\otimes R_{|M}\rightarrow W_{k+1,x}\rightarrow W_{k,x}\rightarrow0.
\]
We may assume that the composition series is consistent with this exact
sequence. This implies that the $\nu_{j}$ that appear in $W_{k}/W_{k+1}$ are
of the form $\mu+\alpha_{1}+...+\alpha_{k}$ with $\alpha_{i}$ a restricted
positive root (i.e. a weight of $\mathfrak{a}$ on $\mathfrak{n}$).

Now consider the corresponding exact sequence of $(\mathfrak{g},K)$--modules.%
\[
\overset{}{(\ast)}0\rightarrow I(\eta_{k,x})\rightarrow I(\sigma
_{k+1,x})\rightarrow I(\sigma_{k,x})\rightarrow0.
\]
The $(\mathfrak{g},K)$--modules $I(\sigma_{\nu})$ with $\sigma$ an irreducible
representation of $M$ with Harish-Chandra parameter $\Lambda_{\sigma}$ (for
$Lie(M)_{\mathbb{C}}$) and $\nu\in\mathfrak{a}_{\mathbb{C}}^{\ast}$ have
infinitesimal character with Harish-Chandra parameter $\Lambda_{\sigma}+\nu$.
We are finally ready to prove the theorem.

Let $C_{\omega}$ be the compact set $\cup_{x\in\omega}ch(J(R_{x}))$. Let
$C_{\omega}=\cup_{j=1}^{k_{\omega}}\Lambda_{i}+D_{i}$ with $D_{i}$ compact in
$\mathfrak{a}_{\mathbb{C}}^{\ast}$. \ Assume that the result is false for
$\omega$. Then for each $j$ there exists $k\geq j$ and $x$ such that $\ker
T_{k}(x)\neq0$ but $\ker T_{k+1}=0$. Label the Harish -Chandra parameters that
appear in $I(\sigma_{o,x})$, $\Lambda_{1}+\nu_{1},...,\Lambda_{s}+\nu_{s}$
with $\Lambda_{i}\in Lie(T)^{\ast}$ and $\nu_{i}\in\mathfrak{a}_{\mathbb{C}%
}^{\ast}$ (recall that we have fixed a maximal torus of $M$). The above
observations imply that $ch(J(R_{x})$ contains an element of the form
$\Lambda+\nu_{i_{k}}+\beta_{k}$ with $\beta_{k}$ a sum of $k$ positive roots,
$\Lambda\in Lie(T)^{\ast}$ and $1\leq i_{k}\leq s$. We now have our
contradiction $\nu_{i_{k}}+\beta_{k}\in\cup D_{i}$ which is compact. But the
set of $\nu_{i_{k}}+\beta_{k} $ is unbounded.

\section{Families of Hilbert and Fr\'{e}chet representations}

\begin{definition}
Let $X$ be metric space. A continuous family of Hilbert representations based
on $X$ of $G$ is a pair $(\pi,H)$ of a Hilbert space $H$ and $\pi:X\times
G\rightarrow H$ strongly continuous such that if $\pi_{x}(g)=\pi(x,g)$ then
$(\pi_{x},H)$ is a strongly continuous representation of $G$. The family will
be called admissible if $\pi_{x|K}$ is independent of $x\in X$ and $\dim
H(\gamma)<\infty$ for each $\gamma\in\hat{K}$.
\end{definition}

\begin{lemma}
Let $(\pi,H)$ be a continuous family of admissible Hilbert representations of
$G$ based on the connected real or complex analytic manifold $X$ and denote by
$d\pi_{x}$the action of $\mathfrak{g}$ on $H_{K}^{\infty}$ (the $K$--finite
$C^{\infty}$--vectors). Then $(d\pi,H_{K})$ is a continuous family of
admissible $(\mathfrak{g},K)$--modules based on $X$.
\end{lemma}

\begin{proof}
If $\gamma\in\hat{K}$ then $C_{c}^{\infty}(\gamma;G)$ denotes the space of all
$f\in C_{c}^{\infty}(G)$ such that
\[
\int_{K}\chi_{\gamma}(k)f(k^{-1}g)dk=f(g),g\in G
\]
with $\chi_{\gamma}$ the character of $\gamma$. Then
\[
H(\gamma)=\pi_{x}(C_{c}^{\infty}(\gamma;G))H.
\]
We also note that if $Y\in\mathfrak{g},f\in C_{c}^{\infty}(\gamma;G)$ and
$v\in H$ then
\[
d\pi_{x}(Y)\pi_{x}(f)v=\pi_{x}(Yf)v
\]
with $Yf$ the usual action of $Y\in\mathfrak{g}$ on $C^{\infty}(G)$ as a left
invariant vector field. Thus, if $v\in H_{K}$ and $y\in U(\mathfrak{g}%
_{\mathbb{C}})$ then the map%
\[
x\longmapsto d\pi_{x}(y)v
\]
is continuous.
\end{proof}

The following lemma is Lemma 1.1.3 in \cite{RRG1-11} taking into account
dependence on parameters. The proof is essentially the same taking into
account the dependence on parameters and using the local compactness of $X.$

\begin{lemma}
Let $X$ be a locally compact metric space and let $H$ be a Hilbert space.
Assume that for each $x\in X$, $\pi_{x}:G\rightarrow GL(H)$ (bounded
invertible operators such that

1) If $\omega\subset X$ and $\Omega\subset G$ are compact subsets of $X$ and
of $G$ respectively then there exists a constant $C_{\omega,\Omega}$ such that
$\left\Vert \pi_{x}(g)\right\Vert \leq C_{\omega,\Omega}$ for $x\in\omega
,g\in\Omega.$

2) The map $x,g\rightarrow\left\langle \pi_{x}(g)v,w\right\rangle $ is
continuous for all $v,w\in H$.

Then $(\pi,H)$ is a continuous family of representations of $G$ based on $X$
and conversely if $(\pi,H)$ is a continuous family of Hilbert representations
then 1) and 2) \ are satisfied.
\end{lemma}

An immediate corollary is

\begin{corollary}
\label{Conj-dual}Let $(\pi,H)$ be an admissible, continuous family of Hilbert
representations of $G$ based on the locally compact metric space $X$. Set for
each $x\in X$, $\hat{\pi}_{x}(g)=\pi_{x}(g^{-1})^{\ast}$ then $(\hat{\pi},H)$
is a continuous, admissible family of Hilbert representations of of $G$ based
on $X$.
\end{corollary}

Let $\left\Vert g\right\Vert $ be a norm on $G,$ that is a continuous function
from \thinspace$G$ to $\mathbb{R}_{>0}$ (the positive real numbers) such that

1. $\left\Vert k_{1}gk_{2}\right\Vert =\left\Vert g\right\Vert ,k_{1},k_{2}\in
K,g\in G$,

2. $\left\Vert xy\right\Vert \leq\left\Vert x\right\Vert \left\Vert
y\right\Vert ,x,y\in G$,

3. The sets $\left\Vert g\right\Vert \leq r<\infty$ are compact.

4. If $X\in\mathfrak{p}$ then if $t\geq0$ then $\log\left\Vert \exp
tX\right\Vert =t\log\left\Vert \exp X\right\Vert .$

If $(\sigma,V)$ is a finite dimensional representation of $G$ with compact
kernel and if $\left\langle ...,...\right\rangle $ is an inner product on $V$
that is $K$--invariant then if $\left\Vert \sigma(g)\right\Vert $is the
operator norm of $\sigma(g)$ then $\left\Vert g\right\Vert =\left\Vert
\sigma(g)\right\Vert $ is a norm on $G$. Taking the representation on $V\oplus
V$ given by%
\[
\left[
\begin{array}
[c]{cc}%
\sigma(g) & \\
& \sigma(g^{-1})^{\ast}%
\end{array}
\right]
\]
then we may (and will) assume in addition

5. $\left\Vert g\right\Vert =\left\Vert g^{-1}\right\Vert $.

Note that 5. implies that $\left\Vert g\right\Vert \geq1$.

Using the same proof as Lemma 2.A.2.1 in \cite{RRG1-11}(which we give for the
sake of completeness) one can prove

\begin{lemma}
\label{loc-umod-growth}If $(\pi,H)$ is a continuous family of Hilbert
representations modeled on $X$ and if $\omega$ is a compact subset of $x$ then
there exists constants $C_{\omega},r_{\omega}$ such that%
\[
\left\Vert \pi_{x}(g)\right\Vert \leq C_{\omega}\left\Vert g\right\Vert
^{r_{\omega}}.
\]

\end{lemma}

\begin{proof}
Let $B_{1}=\{g\in G|\left\Vert g\right\Vert \leq1\}$. Then if $v\in H$ and
$(x,g)\in\omega\times B_{1}$ then $\sup\left\Vert \pi_{x}(g)v\right\Vert
<\infty$ by strong continuity. The principle of uniform boundedness (c.f.
\cite{ReedSimon},III.9,p.81)\ implies that there exists a constant, $R$, such
that $\left\Vert \pi_{x}(g)\right\Vert \leq R$ for $(x,g)\in\omega\times
B_{1}$. Let $r=\log R$. In particular if $k\in K$ then $\left\Vert \pi
_{x}(kg)\right\Vert \leq\left\Vert \pi_{x}(k)\right\Vert \left\Vert \pi
_{x}(g)\right\Vert \leq R\left\Vert \pi_{x}(g)\right\Vert $. Also,
\[
\left\Vert \pi_{x}(g)\right\Vert =\left\Vert \pi_{x}(k^{-1})\pi_{x}%
(kg)\right\Vert \leq R\left\Vert \pi_{x}(kg)\right\Vert .
\]
Thus for all $k\in K,g\in G$%
\[
R^{-1}\left\Vert \pi_{x}(g)\right\Vert \leq\left\Vert \pi_{x}(kg)\right\Vert
\leq R\left\Vert \pi_{x}(g)\right\Vert .
\]
Let $X\in\mathfrak{p}$, $X\neq0$ and let $j$ be such that
\[
j<\log\left\Vert \exp X\right\Vert \leq j+1
\]
then%
\[
\log\left\Vert \pi_{x}(\exp X)\right\Vert \leq\log\left\Vert \pi_{x}%
(\exp(\frac{X}{j+1})\right\Vert ^{j+1}\leq r(j+1)\leq r+r\log\left\Vert \exp
X\right\Vert .
\]
Thus%
\[
\left\Vert \pi_{x}\left(  \exp X\right)  \right\Vert \leq R\left\Vert \exp
X\right\Vert ^{r},X\in\mathfrak{p.}%
\]
If $g\in G$ then $g=k\exp X$ with $k\in K$ and $X\in\mathfrak{p}$ so%
\[
\left\Vert \pi_{x}(g)\right\Vert =\left\Vert \pi_{x}(k\exp X)\right\Vert \leq
R^{2}\left\Vert \exp X\right\Vert ^{r}=R^{2}\left\Vert g\right\Vert ^{r}.
\]
Take $C_{\omega}=R^{2}$ and $r_{\omega}=r$.
\end{proof}

We define $\mathcal{S}(G)$ to be the space of all $f\in C^{\infty}(G)$ such
that of $x\in U(\mathfrak{g}_{\mathbb{C}})$ (thought of as a left invariant
differential operator) and $r>0$ then%
\[
p_{r,x}(f)=\sup_{g\in G}\left\vert xf(g)\right\vert \left\Vert g\right\Vert
^{r}<\infty.
\]
$\mathcal{S}(G)$ is a Fr\'{e}chet (using the semi-norms $p_{r,x}$) algebra
(under convolution) of functions on $G.$

Lemma 2.A.2.4 in \cite{RRG1-11} implies that there exists $d>0$ such that%
\[
\int_{G}\frac{dg}{\left\Vert g\right\Vert ^{d}}<\infty.
\]
This implies that $\mathcal{S}(G)$ acts on any Banach representation,
$(\pi,V)$ of $G$ via%
\[
\pi(f)=\int_{G}f(g)\pi(g)dg.
\]

Recall that a pair $(\pi,V)$ of a Fr\'{e}chet space, $V$, and a representation
of $G$, $\pi$, on $V$ is called a smooth Fr\'{e}chet representation of
moderate growth if the map $g\longmapsto\pi(g)v$ is $C^{\infty}$ and if $p$ is
a continuous seminorm on $V$ then there exists a continuous seminorm $q$ on
$V$ and $r$ such that%
\[
p(\pi(g)v)\leq\left\Vert g\right\Vert ^{r}q(v)\text{.}%
\]
This implies that a smooth Fr\'{e}chet representation of moderate growth is an
$\mathcal{S}(G)$--module. A smooth Fr\'{e}chet representation of moderate
growth is defined to be admissible if the $(\mathfrak{g},K)$--module $V_{K}$
is admissible. It is said to be of Harish-Chandra class if $V_{K}$ is
admissible and finitely generated. Let $\mathcal{HF}(G)$ be the category of
smooth Fr\'{e}chet representations of moderate growth in the Harish-Chandra class.

The CW theorem

\begin{theorem}
The functor $V\rightarrow V_{F}$ from $\mathcal{HF}(G)$ to $H\mathcal{(}%
\mathfrak{g},K)$ is an isomorphism of categories.
\end{theorem}

We will prove this as a consequence of the usual statement of the theorem is
(see \cite{RRG1-11} Theorem 11.6.7 (2))

\begin{theorem}
If $(\pi_{i},V_{i})\in\mathcal{HF}(G)$, for $i=1,2$ and if $T\in
\mathrm{Hom}_{H\mathcal{(}\mathfrak{g},K)}(\left(  V_{1}\right)  _{K},\left(
V_{2}\right)  _{K})$ then $T$ extends to a continuous element of
$\mathrm{Hom}_{\mathcal{FH(}G)}(V_{1},V_{2})$ with closed image that is a
topological summand.
\end{theorem}

If $V_{1},V_{2}\in\mathcal{HF}(G)$ have the property that $\left(
V_{1}\right)  _{K}=\left(  V_{2}\right)  _{K}=V$ then one has
\[
V_{i}\subset\prod_{\gamma\in\hat{K}}V(\gamma),i=1,2.
\]
As the formal sums that converge relative to the continuous seminorms endowing
the topology on $V_{1}$ and $V_{2}$ respectively. The identity map on $V$
induces an isomorphism of $V_{1}$ and $V_{2}$. But this is given by the
identity map on $\prod_{\gamma\in\hat{K}}V(\gamma)$. Hence $V_{1}=V_{2}$. This
implies the isomorphism of categories.

The inverse functor can be seen as follows. Let $V\in$ $H\mathcal{(}%
\mathfrak{g},K)$ and let $(\pi,H)$ be a Hilbert representation of $G$ such
that $(d\pi,H_{K})$ is equivalent to $V$. Let $T\in\mathrm{Hom}_{H\mathcal{(}%
\mathfrak{g},K)}(V,H_{K})$ give the isomorphism. Let $\left\langle
...,...\right\rangle $ the Hilbert space structure on $H$ and let $\left(
v,w\right)  =\left\langle Tv,Tw\right\rangle $. If $x\in U(\mathfrak{g)}$ set
$p_{x}(v)=\sqrt{(xv,xv)}$and
\[
\overline{V}=\{\{v_{\gamma}\}\in\prod_{\gamma\in\hat{K}}V(\gamma)|\sum
_{\gamma\in\hat{K}}p_{x}(v_{\gamma})^{2}<\infty\}.
\]
Then $T$ extends to an isomorphism of $\bar{V}$ onto $H^{\infty}$. Thus
defining $\mu(g)=T^{-1}\pi(g)T$ on $\bar{V}$ we have $(\mu,\bar{V}%
)\in\mathcal{HF}(G)$ and $\bar{V}_{K}=V$. The uniqueness implies that
$V\rightarrow\bar{V}$ defines the inverse functor.

Another corollary of the CW theorem is (see \cite{HOW} Theorem 11.8.2)

\begin{theorem}
If $(\pi,V)\in\mathcal{HF}(G)$ and if $v\in V$ then $\pi(\mathcal{S}(G))v$ is
closed in $V$ and a topological summand.
\end{theorem}

\begin{corollary}
If $(\pi,H)$ is a Hilbert representation of $G$ such that $H_{K}^{\infty}\in
H\mathcal{(}\mathfrak{g},K)$ and if $H_{K}$ is generated by the subspace $U$
then $\pi(\mathcal{S}(G))U=H^{\infty}$.
\end{corollary}

\begin{definition}
A continuous family of objects in $\mathcal{HF}(G)$ based on the metric space
$X$ is a pair $(\pi,V)$ of a Fr\'{e}chet space $V$ and a continuous map
\[
\pi:X\times G\rightarrow\mathrm{End}(V)
\]
(here $\mathrm{End}(V)$ is the algebra of continuous operators on $V$ with the
strong topology) such that such that for each $x\in X$, if $\pi_{x}%
(g)=\pi(x,g)$ then $(\pi_{x},V)\in\mathcal{HF}(G)$. We will say that the
family has local uniform moderate growth if for each $\omega$ a compact subset
of $X$ and each continuous seminorm on $V,p,$there exists a continuous
seminorm $q_{\omega}$ on $V$ and $r_{\omega}$ such that if $v\in V $ then
\[
p(\pi_{x}(g)v)\leq q_{\omega}(v)\left\Vert g\right\Vert ^{r_{\omega}}.
\]

\end{definition}

\begin{definition}
\label{holomorphicfam}A holomorphic family of objects in $\mathcal{HF}(G)$
based on the complex manifold $X$ is a continuous family $(\pi,V)$ such that
the map $x\longmapsto\pi_{x}(g)v$ is holomorphic from $X$ to $V$ for all $g\in
G$,$v\in V$.
\end{definition}

\begin{lemma}
\label{Cont-diff}If $(\pi,H)$ is a continuous family of Hilbert
representations based on the metric space $X$ such that the representations
$(d\pi_{x},H_{K}^{\infty})\in H\mathcal{(}\mathfrak{g},K)$ and the
$K-C^{\infty}$vectors are the $G-C^{\infty}$ vectors then $(\pi,H^{\infty})$
is a continuous family of of objects in $\mathcal{HF}(G)$ based on the metric
space $X$ that is of local uniform moderate growth.
\end{lemma}

\begin{proof}
We note that if $f\in\mathcal{S}(G)$ and $v\in H$ then the map $x\longmapsto
\pi_{x}(f)v$ is continuous from $X$ into $H^{\infty}$. Also $\pi_{x}(h)\pi
_{x}(f)v=\pi_{x}(L(h)f)v$ with $L(h)f(g)=f(h^{-1}g)$. The last assertion
follows from Lemma \ref{loc-umod-growth}.
\end{proof}

For want of a better place to put it we include the following simple Lemma in
this section.

\begin{lemma}
\label{orthhbases}Let $(\tau,V)$ be a finite dimensional continuous
representation of $K$ and let $X$ be a locally compact metric space (resp. an
analytic manifold). If $u\in X$ let $\left\langle ...,...\right\rangle _{u}$
be an inner product on $V$ such that $\tau(k)$ acts unitarily with respect to
$\left\langle ...,...\right\rangle _{u}$ for $k\in K$ and such that the map
$u\longmapsto\left\langle v,w\right\rangle _{u}$ is continuous (resp. real
analytic) for all $v,w\in V$. Then there exists, for each $u$ and an ordered
orthonormal basis of $V,$ $e_{1}(u),...,e_{n}(u)$ such that the map
$u\longmapsto e_{i}(u)$ is continuous (resp. real analytic) and the matrix of
$\tau(k)$ with respect to $e_{1}(u),...,e_{n}(u)$ is independent of $u$.
Furthermore, if $X$ is compact and contractible and $(\sigma.W)$ is a finite
dimensional continuous representation of $K$ and $u\longmapsto B(u)\in
\mathrm{Hom}_{K}(V,W)$ is continuous and surjective for $u\in X$ then
$e_{1}(u),...,e_{r}(u)$ with $r=\dim V-\dim W$ can be taken in $\ker B(u)$.
\end{lemma}

\begin{proof}
Fix an inner product, $(...,...)$, on $V$ such that $\tau$ is unitary. Then
there exists a positive definite Hermitian operator (with respect to $\left(
...,...\right)  $), $A(u)$ such that $\left\langle v,w\right\rangle
_{u}=(A(u)v,w),v,w\in V$. Then $A(u)$ is continuous (resp. real analytic) in
$u$. Now,
\[
\left\langle v,w\right\rangle _{u}=\left\langle \tau(k)v,\tau(k)w\right\rangle
_{u}=(A(u)\tau(k)v,\tau(k)w)=(\tau(k)^{-1}A(u)\tau(k),v,w\in V,k\in K.
\]
So
\[
\tau(k)^{-1}A(u)\tau(k)=A(u),u\in X,k\in K.
\]
Set $S(u)=A(u)^{\frac{1}{2}}$ then $\left\langle v,w\right\rangle
_{u}=(S(u)v,S(u)w)$. Thus if $T(u)=S(u)^{-\frac{1}{2}}$ then $\tau
(k)T(u)=T(u)\tau(k),k\in K,u\longmapsto T(u)$ is continuous (resp. real
analytic) and
\[
\left\langle T(u)v,T(u)w\right\rangle _{u}=\left(  v,w\right)  ,v,w\in V.
\]
Let $e_{1},...,e_{n}$ be an (ordered) orthonormal basis of $V$ with respect to
$(...,...)$ then $e_{1}(u)=T(u)e_{1},...,e_{n}(u)=T(u)e_{n}$ is an orthonormal
basis of $V$ with respect to $\left\langle ...,...\right\rangle _{u}$ .If
$\tau(k)e_{i}=\sum k_{ji}e_{j}$ then%
\[
\tau(k)e_{i}(u)=\tau(k)T(u)e_{i}=T(u)\tau(k)e_{i}=\sum k_{ji}T(u)e_{j}.
\]
To prove the second assertion note that $u\rightarrow\ker B(u)$ is a
$K$--vector bundle over $X$. Since $X$ compact and contractible the bundle is
a trivial $K$--vector bundle (\cite{Atiya},Lemma 1.6.4). Thus there is a
representation $(\mu,Z)$ of $K$ and $u\longmapsto L(u)\in Hom_{K}(Z,V)$
continuous such that $L(u)Z=\ker B(u)$ and $L(u)$ is injective. Notice that
$B(u):\ker B(u)^{\perp}\rightarrow W$ is a $K$--module isomorphism. Now pull
back the inner product $\left\langle _{...,...}\right\rangle _{u}$ to $Z$
using $L(u)$ getting a $K$--invariant inner product, $(...,,,,)_{u}$, on $Z$
and push the inner product to $W$ getting \ a $K$--invariant inner product
$(...,...)_{u}^{1}$ on $W$ Now apply the first part of the lemma to get an
orthonormal basis $f_{1}(u),...,f_{r}(u)$ of $Z$ with respect to
$(...,...)_{u}$ and an orthonormal basis $f_{r+1}(u),...,f_{n}(u)$ ($n=\dim
V$) with respect to $(...,...)_{u}^{1}$ such that the matrices of the action
of $K$ with respect to each of these bases is constant. Take $e_{i}%
(u)=L(u)f_{i}(u)$ for $i=1,...,r$ and $e_{i}(u)=\left(  B(u)_{|_{\ker
B(u)^{\perp}}}\right)  ^{-1}f_{i}(u)$ for $i=r+1,...,n$.
\end{proof}

\section{Continuous globalization of families of $J$--modules}

We maintain the notation of the previous sections.

Let $Z$ be a connected analytic manifold and let $(\mu,L)$ be an analytic
family of of objects in $W(K,\mathbf{D})$ based on $Z$.

\begin{theorem}
\label{globalize-J}Let $U\subset Z$ be open with compact closure. There exists
a continuous family $(\pi,H)$ of Hilbert representations of $G$ based on $U$
such that the continuous family of $(\mathfrak{g},K)$--modules $(d\pi
,H_{K}^{\infty})$ is isomorphic with the analytic family $z\mapsto J(L_{z})$
of objects in $H\mathcal{(}\mathfrak{g},K)$ based on on $U$ (thought of as a
continuous family). Furthermore, the $K-C^{\infty}$ vectors of $(\pi_{u},H)$
are the $G-C^{\infty}$ vectors for every $u\in U.$
\end{theorem}

\begin{proof}
Let $\gamma\in\hat{K}$ then Theorem \ref{decompositions} 2.implies
\[
\dim J(L_{z})(\gamma)=\dim E\dim\gamma\dim\mathrm{Hom}_{K}(V_{\gamma
},\mathcal{H}\otimes L).
\]
for every $z\in Z$. In particular it is independent of $z$. Theorem
\ref{Imbedding} implies that there exists $k$ and for each $u\in U$ the map
\[
T_{k,L_{u}}:J(L_{u})\rightarrow I(\sigma_{k,L_{u}})
\]
is injective. Note that the space of $K$--finite vectors in $I(\sigma
_{k,L_{u}})$ is the $K$--finite induced representation $Ind_{M}^{K}%
(\sigma_{k,L_{|M}})$ and hence independent of $u$. Let \thinspace
$(H_{1},\left\langle ...,...\right\rangle )$ be the Hilbert space completion
of $Ind_{M}^{K}(\sigma_{k,L_{|M}})$ corresponding to unitary induction from
$M$ to $K$. This gives an analytic family of Hilbert representations of $G$,
$\mu_{z}$. For each $\gamma\in\hat{K}$ the family of linear operators
$T_{k,L_{u}|J(L_{u})(\gamma)}\in\mathrm{Hom}_{K}(J(L_{u})(\gamma
),I(\sigma_{k,L_{u}})(\gamma))$ is analytic in $u\in U$ (see Theorem
\ref{J-morph}) and injective. On $\left(  E\otimes\mathcal{H\otimes}L\right)
(\gamma)$ put for each $u\in U$ the inner product $\left\langle
...,...\right\rangle _{u}=T_{k,L_{u}}^{\ast}\left\langle ...,...\right\rangle
$. Then $\pi(k)$ acts unitarily with respect to $\left\langle
...,...\right\rangle _{u}$ for $u\in U$ and $u\longmapsto\left\langle
...,...\right\rangle _{u}$ is real analytic. Let $e_{1}^{\gamma}%
(u),...,e_{r_{\gamma}}^{\gamma}(u)$ be as in Lemma \ref{orthhbases}. Then
$u\longmapsto e_{i}^{\gamma}(u)$ is analytic. Put $f_{i}^{\gamma
}(u)=T_{k,L_{u}}e_{i}^{\gamma}(u)$ and set $P(u)_{\gamma}(v)=\sum\left\langle
v,f_{i}^{\gamma}(u)\right\rangle f_{i}^{\gamma}(u)$ then $u\longmapsto
P(u)_{\gamma}$ is an analytic map of $U$ into the manifold of orthogonal
projection operators of rank $m_{\gamma}$ on $H_{1}(\gamma)$. Set
\ref{globalize-J}%
\[
P(u)=\sum_{\gamma\in\hat{K}}P(u)_{\gamma}.
\]
Then $P(u)H_{1}$ is the closure in the Hilbert space $H_{1}$ of $T_{k,L_{u}%
}(J(L_{u}))$. Observing that the $K$--finite vectors in $(\mu_{u},H_{1})$ are
contained in the analytic vectors implies that $P(u)H_{1}$ is invariant under
$\mu_{u}(G)$. Note $u\mapsto P(u)$ is continuous in the strong operator
topology from $U$ to the bounded operators on $H_{1define}$ . This is proved
by the following standard calculus style argument. Let $v\in H_{1}$ be a unit
vector and $u_{o}\in U$. We can expand $v=\sum v_{\gamma}$ in $H_{1}$ with
$v_{\gamma}\in H(\gamma)$ and let $\varepsilon>0$ be given then there exists
$F\subset\hat{K}$ a finite set such that $\left\Vert \sum_{\gamma\notin
F}v_{\gamma}\right\Vert <\frac{\varepsilon}{4}$. Also $P(u)_{F}=\sum
_{\gamma\in F}P(u)_{\gamma}$ is analytic in $u$ thus there exists a
neighborhood $U_{1}$ of $z_{o}$ in $U$ such that $\left\Vert \left(
P(u)_{F}-P(u_{o})_{F}\right)  v\right\Vert <\frac{\varepsilon}{2}$.for $u\in
U_{1}$. Noting that $\left\Vert P(u)\right\Vert =1$ we have $\left\Vert
\left(  P(u)-P(u_{o})\right)  v\right\Vert \leq\left\Vert \left(
P(u)_{F}-P(u_{o})_{F}\right)  v\right\Vert +\left\Vert \left(  P(u)-P(u_{o}%
)\right)  \sum_{\gamma\notin F}v_{\gamma}\right\Vert <$%
\[
\left\Vert \left(  P(u)_{F}-P(u_{o})_{F}\right)  v\right\Vert +\frac
{\varepsilon}{2}<\varepsilon.
\]

For each $u\in U$ put the inner product $\left\langle ...,...\right\rangle
_{\upsilon}=T_{u,L_{o}}^{\ast}\left\langle ...,...\right\rangle $ on
$E\otimes\mathcal{H\otimes}L$. Pull back the action of $G$ on $P(u)H_{1}$ to
the Hilbert space completion, $H_{u}$ of $E\otimes\mathcal{H\otimes}L$ with
respect to $\left\langle ...,...\right\rangle _{\upsilon}$ to get the
representation $\eta_{u}$ of $G$ such that $d\eta_{u}$ is equivalent with
$J(L_{u})$. Note that $\{e_{j}(u)\}$ is an orthonormal basis of $H_{u}$ for
all $u\in Z$. If $v,u\in Z$ define $T(v,u):H_{u}\rightarrow P(v)H_{1}$ by
$T(v,u)e_{j}(u)=e_{j}(v)$. Then $T(u,v)$ is a unitary $K$--isomorphism with
inverse $T(v,u)$. Fix $u_{o}\in U$, set $H=H_{u_{o}}$ and set $\pi
_{u}(g)=T(u,u_{o})\eta_{u}(g)_{|}T(u_{o},u).$ Then $(\pi,H)$ is the desired
continuous family. The last assertion follows from the fact that the
$K-C^{\infty}$ vectors of $(\mu_{u},H_{1})$ are the $G-C^{\infty}$ vectors.
\end{proof}

The technique in the proof of the Theorem above involving the bases
$\{e_{j}^{\gamma}(u)\}$ will be used several times in the next section.

\section{Continuous globalization of families of objects in $H(\mathfrak{g}%
,K)$}

\begin{theorem}
\label{hilb-main}Let $(\pi,V)$ be an analytic family of objects in
$H(\mathfrak{g},K)$ based on the analytic manifold $X$. Let $\ x_{o}\in X$
then there exists, $U$, an open neighborhood of $x_{o}$ in $X$ and a
continuous family of Hilbert representations $(\mu_{U},H_{U})$ such that the
family $(d\mu_{U},\left(  H_{U}\right)  _{K}^{\infty})$ is isomorphic with
$(\pi_{|U},V)$(as a continuous family). Furthermore, the $K-C^{\infty}$
vectors of $\mu_{U,x}$ are the $G-C^{\infty}$ vectors.
\end{theorem}

\begin{proof}
Let $U_{1}$ be an open neighborhood of $x_{o}$ in $X$ with compact closure.
Then Theorem \ref{loc-finite} implies that there exists $F_{U_{1}}^{0}%
\subset\hat{K}$ a finite subset such that $\pi_{x}(U(\mathfrak{g}_{\mathbb{C}%
}))\sum_{\gamma\in F_{U}^{0}}V(\gamma)=V.$Let $R^{0}=\sum_{\gamma\in F_{U}%
}V(\gamma)$. $R^{0}$ is invariant under the action $\pi_{x}(\mathbf{D)}$ for
all $x\in X$. This implies that $(\left(  \pi_{|U}\right)  |_{\mathbf{D}%
},R^{0})$ defines an analytic family of objects in $W(K,\mathbf{D})$ based on
$U_{1}$. Let $J(R^{0})$ be the corresponding $J$--family. Then we have the
surjective analytic homomorphism of families%
\[%
\begin{array}
[c]{ccccc}
& T_{0} &  &  & \\
J(R^{0}) & \rightarrow & V_{|U} & \rightarrow & 0
\end{array}
\]
with $T_{0}(x)$ mapping $J(R_{x}^{0})$ onto $V$ for all $x\in U_{1}$. Let
$\mu_{x}^{0}$ be the action of $U(\mathfrak{g}_{\mathbb{C}})$ on the space
$J(R^{0})$ (which we regard to be the fixed $K$--representation $\mathcal{H}%
\otimes E\otimes R^{0}$) the Corollary \ref{finitegen} implies that there
exists a finite subset $F_{U}^{1}\subset\hat{K}$ such that
\[
\ker T_{0}(x)=\mu_{x}(U(\mathfrak{g}_{\mathbb{C}}))\sum_{\gamma\in F_{U}^{1}%
}\ker T_{0}(x)|_{J(R^{0})(\gamma)}.
\]
If $\gamma\in\hat{K}$ then
\[
\dim\ker T_{0}(x)|_{J(R^{0})(\gamma)}=\dim J^{0}(R)(\gamma)-\dim V(\gamma)
\]
for $x\in U_{1}.$ Let $(\sigma,(H^{0},(...,...)))$ be the continuous family of
Hilbert representations based on $U_{1}$ corresponding to $J(R^{0})$ as in
Theorem \ref{globalize-J}. Let $U$ be an open neighborhood of $x_{o}$
contained in $U$ such that $\overline{U}$ is contractible. Let $\gamma\in
\hat{K}$ if $x\in\overline{U}$ let $e_{1}^{\gamma}(x),...,e_{n_{\gamma}}%
^{y}(x)$ and orthonormal basis of $\left(  \mathcal{H}\otimes E\otimes
R^{0}\right)  (\gamma)$ with respect to the pull back of $(...,...)$ to
$\left(  \mathcal{H}\otimes E\otimes R^{0}\right)  (\gamma)$ such that
$x\mapsto e_{i}^{\gamma}(x)$ is continuous on $U$ and $e_{1}^{\gamma
}(x),...,e_{r_{\gamma}}^{\gamma}(x)$ is an orthonormal basis of $\ker
T_{1}(x)|_{\left(  \mathcal{H}\otimes E\otimes R^{0}\right)  (\gamma)}$ and
the matrix of $k\in K$ withe respect to the $e_{i}^{\gamma}(x)$ is independent
of $x$ (see the second part of Lemma \ref{orthhbases}) for $x\in U$. Define
\[
f_{i}^{\gamma}(x)=T_{0}(x)(e_{r_{\gamma}+i}^{\gamma}(x)),i=1,...,\dim
V(\gamma)\text{.}%
\]
Let if $x\in U$ let $\left\langle ...,...\right\rangle _{x}$ be the inner
product on $V$ such that $\{f_{i}^{\gamma}(x)\}_{\gamma,i}$ is an orthonormal
basis of $V$. Define $H_{x}$ to be the Hilbert space completion of $V$ with
respect to $\left\langle ...,...\right\rangle _{x}$. Let $H_{x}^{0}$ be the
closure of $\ker T_{0}(x)$ in $H$. then since $\ker T_{x}(x)$ is a
$(\mathfrak{g},K)$ invariant subspace of the analytic vectors $H_{x}$ is
$\sigma(G)$--invariant (c.f. \cite{RRG1-11} Proposition 1.6.6). The argument
using the $e_{i}^{\gamma}$ in the proof of Theorem \ref{globalize-J} one
proves that $(\sigma_{x|H_{x}^{0}},H_{x}^{0})$ defines a continuous family of
Hilbert representations based on $U$. Also the space of $K$--finite vectors of
$H^{0}/H_{x}^{0}$ is isomorphic with $(\pi_{x},V)$. Let $\mu_{x}$ be the
quotient representation on $H^{0}/H_{x}^{0}$. Note that $\,$the quotient map
on $J(R^{0})/\ker T_{0}(x)$ corresponding to $T_{0}(x),S(x),$ extends to a
unitary map of $H^{0}/H_{x}^{0}$ onto $H_{x}$ by the definition of
$\left\langle ...,...\right\rangle _{x}$. Set $\eta_{x}(g)=S(x)\mu
_{x}(g)S(x)^{-1}$ for $x\in U$. Finally, if $x,y\in U$ then define
$L(x,y):H_{y}\rightarrow H_{x}$ by
\[
L(x,y)f_{i}^{\gamma}(y)=f_{i}^{\gamma}(x)
\]
for all $i,\gamma$. Then $x,y\rightarrow L(x,y)$ is unitary, strongly
continuos and $L(x,y)^{-1}=L(y,x)$. Set $H=H_{x_{o}}$ and $\nu_{x}%
(g)=L(x,x_{o})\eta_{x}(g)L(x_{o},x$) \ then $(\nu,H)$ is the desired
continuous family of Hilbert representations based on $U$.
\end{proof}

We include the following corollary however as noted at the end of the section
it is not necessary to prove the main results that follow.

\begin{corollary}
Let $X$ be a connected analytic manifold and let $(\pi,V)$ be an analytic
family of objects in $H(\mathfrak{g},K)$ based on $X$ such that $\dim
V=\infty$ then there exists a continuous family of Hilbert representations
$(\lambda,H)$ such that the family $(d\lambda,\left(  H\right)  _{K}^{\infty
})$ is isomorphic with $(\pi,V)$(as a continuous family). Furthermore, the
$K-C^{\infty}$ vectors of $\mu_{U,x}$ are the $G-C^{\infty}$ vectors.
\end{corollary}

\begin{proof}
The previous theorem implies that there exists an open covering $\{U_{\alpha
}\}$ of $X$ \ such that for each $\alpha$ there exists a continuous family of
Hilbert representations $(\mu_{U_{\alpha}},H_{U_{\alpha}})$ such that
$(d\mu_{U_{\alpha}},(H_{U_{\alpha}})_{K}^{\infty})$ is continuously isomorphic
with $(\pi|_{U},V)$. For each $\alpha,\beta$ the definition of the Hilbert
spaces $H_{U_{\alpha}}$ implies that one has $g_{U_{\beta},U_{\alpha}%
}(x):H_{U_{\alpha}}\rightarrow H_{U_{\beta}}$ a unitary isomorphism depending
strongly continuously on $x\in U_{\alpha}\cap U_{\beta}$. \ This defines a
Hilbert vector bundle over $X$. Kuiper's Theorem implies that all Hilbert
bundles with infinite dimensional fibers are trivial (\cite{Boss-Bleek},
p.67). Thus there exists a fixed Hilbert space, $H$, and for each $\alpha$ and
each $x\in U_{\alpha},h_{\alpha}(x):H_{U_{\alpha}}\rightarrow H$ a unitary
isomorphism that depends strongly continuously on $x$ such that $g_{U_{\beta
},U_{\alpha}}(x)=h_{\beta}(x)^{-1}h_{\alpha}(x)$. Define $\lambda_{x}%
(g)=h_{U}(x)\mu_{x}(g)h_{U}(x)^{-1}$.
\end{proof}

This and Lemma \ref{Cont-diff} imply our main results

\begin{theorem}
\label{main}There exists a continuous family $(\lambda,Z)$ of objects in
$\mathcal{HF}(G)$ based on $X$ of local uniform moderate growth that
globalizes the family $(\pi,V)$.
\end{theorem}

This can be interpreted in the following way:

\begin{corollary}
\label{main2}Let $T$ be the inverse functor to the $K$--finite functor
$\mathcal{HF}(G)\rightarrow H(\mathfrak{g},K)$ and let $(\pi,V)$ is an
analytic family of objects in $H(\mathfrak{g},K)$ based on the connected
analytic manifold $X$ such that $\dim V=\infty$. If $T((\pi_{x},V))=(\lambda
_{x},\overline{V_{x}})$ then

1. For all $x,y\in X,\overline{V_{x}}=\overline{V_{y}}$ as subspaces of
$\prod_{\gamma\in\hat{K}}V(\gamma)$. Set $\overline{V}$ equal to the common value.

2. The map $x,g,v\longmapsto\lambda_{x}(g)v$ is continuous from $X\times
G\times\overline{V}$ to $\bar{V}$, linear in $v$ and $C^{\infty}$ in $g$.
\end{corollary}

With this interpretation it is clear that this result follows from the local
version of the Hilbert globalization (that is the first theorem in this section).

\section{The dual functor}

We now consider a dual functor. Let $(\pi,V)\in\mathcal{HF}(G)$ let
$\lambda\in V^{\prime}$ (the continuous dual). If $v\in V_{K}$ then the
following assertions are true

1. There exists $f_{\lambda,v}$ a real analytic function on $G$ such that
$f_{\lambda,v}(1)=\lambda(v)$.

2. If $R_{g}$ denotes the right regular action of $g\in G$ on $C(G)$ then
$R_{k}f_{\lambda,v}=f_{\lambda,kv}$ for $k\in K$. If $x\in U(\mathfrak{g})$ is
thought of as a left invariant differential operator then $xf_{\lambda
,v}=f_{\lambda,xv}$.

3. There exists $d$ depending only on $\lambda$ and $C_{v}>0$ such that
$\left\vert f_{\lambda,v}(g)\right\vert \leq C_{v}\left\Vert g\right\Vert ^{d}
$

We note that conditions 1. and 2 uniquely specify $f_{\lambda,v}=\lambda
(\pi(g)v)$ which satisfies 3.

If $Z$ is an object of $H(\mathfrak{g},K)$ then denote by
$Z_{\operatorname{mod}}^{\ast}$ set of $\lambda$ in the algebraic dual,
$Z^{\ast}$, \ of $Z$ such that if $v\in Z$ then there exists a real analytic
function $f_{\lambda,v}$ satisfying 1.,2. and 3.

A variant of the CW theorem proved in \cite{RRG1-11}Theorem 11.6.6, Corollary
11.6.3 is

\begin{theorem}
\label{CW-dual}Let $Z\in H(\mathfrak{g},K)$ then if $V\in\mathcal{HF}(G)$ and
if $V_{K}=Z$ then $V_{|Z}^{\prime}=Z_{\operatorname{mod}}^{\ast}.$
\end{theorem}

The purpose of this section is to prove a version of this theorem depending on
parameters for parabolic induced representations. Before we state our result
we will need some definitions and a lemma that will be critical to the proof
of the main result of the section.

Let $P=MAN$ be a real parabolic subgroup with given $K$ standard Langlands
decomposition. Set $M_{K}=M\cap K$. Let $(\sigma,H_{\sigma})$ be an
admissible, finitely generated Hilbert representation of $M$ with Hilbertspace
structure $\left\langle ...,...\right\rangle _{\sigma}$. Let $A=A_{1}A_{2}$
set $M_{1}=MA_{1}$ and let $\mu$ be a unitary character of $A_{1}.$ Then
$\sigma_{1}(ma)=\mu(a)\sigma(m)$ defines a unitary representation of $M_{1}$
on $H_{\sigma}$. Let $I_{\sigma}^{\infty}$ be the $C^{\infty}$ induced
representation of $\left(  \sigma_{|K_{M}},H_{\sigma}^{\infty}\right)  $
\ from $K_{M}$ to $K$. Let $\mathfrak{a}_{2}=Lie(A_{2})$ and let $\nu
\in\left(  \mathfrak{a}_{2}\right)  _{\mathbb{C}}^{\ast}$. If $f\in I_{\sigma
}^{\infty}$ then set%
\[
f_{\nu}(na_{2}m_{1}k)=a_{2}^{\rho}\sigma_{1}(m)f(k),a_{2}\in A_{2},n\in N,m\in
M_{1},k\in K.
\]
since the ambiguity in this expression is in $M_{K}$ this defines a smooth
function from $G$ to $H_{\sigma}$. As usual, define%
\[
\left(  \pi_{\nu}(g)f\right)  (k)=f_{\nu}(kg)\text{.}%
\]
Then $(\pi_{\nu},I_{\sigma}^{\infty})$ defines a family of smooth Fr\'{e}chet
representations of moderate growth. Set $I_{\sigma}=\left(  I_{\sigma}%
^{\infty}\right)  _{K}.$

If $f_{1},f_{2}\in I_{\sigma}^{\infty}$ define
\[
\left\langle f_{1},f_{2}\right\rangle =\int_{K}\left\langle f_{1}%
(k),f_{2}(k)\right\rangle _{\sigma}dk.
\]
Let $\left\Vert ...\right\Vert $ be the corresponding norm on $I_{\sigma
}^{\infty}$ \ and let $H$ be the Hilbert space completion of $I_{\sigma
}^{\infty}$. In ??? se showed that if $X$ is a compact subset of $\left(
\mathfrak{a}_{2}\right)  _{\mathbb{C}}^{\ast}$ Then there exist $C_{X}$ and
$r_{X}$ such that$\left\Vert \pi_{x}(g)\right\Vert \leq C_{X}\left\Vert
g\right\Vert ^{r_{X}}.$

Let $d_{o}$ be such that%
\[
\int_{G}\left\Vert g\right\Vert ^{-d_{o}}dg<\infty.
\]

Fix $U$ open in $\left(  \mathfrak{a}_{2}\right)  _{\mathbb{C}}^{\ast}$ with
compact closure and set $r_{U}=r_{\bar{U}}$.

Let $W$ be a finite dimenional, $K$ and $Z(\mathfrak{g}$\thinspace$)$
invariant subspace such that
\[
d\pi_{\nu}(U(\mathfrak{g}))W=I_{\sigma}%
\]
for all $\nu\in\bar{U}$. Let $f_{1},...,f_{d}$ be an orthonormal basis of $W$
with resepect to $\left\langle ...,...\right\rangle $ (Theorem
\ref{loc-finite}).

If $\nu\in\bar{U}$ define a new inner product on $H$ by
\[
\left\langle u,w\right\rangle _{\nu,s}=\sum_{i}\int_{G}\left\langle \pi_{\nu
}(g)f_{i},w\right\rangle \left\langle u,\pi_{\nu}(g)f_{i}\right\rangle
\left\Vert g\right\Vert ^{-2s-d_{o}}dg.
\]

\begin{lemma}
\label{keylemma}Fix $\nu_{o}$ in $\bar{U}$ then there exists $\xi>0$ such that
if $\nu\in\bar{U}$ and $s>r_{U}+\xi$ then there exists $L>0$ such that if
$u\in I_{\sigma}^{\infty}$%
\[
\left\Vert u\right\Vert _{\nu_{o},s}\leq L\left\Vert u\right\Vert _{\nu,s-\xi
}.
\]

\end{lemma}

\begin{proof}
First note that
\[
\left\Vert u\right\Vert _{\nu_{o},s}^{2}=\sum_{i}\int_{G}\left\vert
\left\langle \pi_{\nu}(g)f_{i},u\right\rangle \right\vert ^{2}\left\Vert
g\right\Vert ^{-2s-d_{o}}dg.
\]
Now%
\[
\left\langle \pi_{\nu_{o}}(g)f_{i},u\right\rangle =\int_{K}\left\langle
\left(  f_{i}\right)  _{\nu_{o}}(kg),u(k)\right\rangle _{\sigma}dk.
\]%
\[
=\int_{K}a(kg)^{\nu-\nu_{o}}\left\langle \left(  f_{i}\right)  _{\nu
}(kg),u(k)\right\rangle _{\sigma}dk.
\]
Thus%
\[
\left\vert \left\langle \pi_{\nu_{o}}(g)f_{i},u\right\rangle \right\vert
^{2}\leq\left(  \int_{K}a(kg)^{2\operatorname{Re}(\nu-\nu_{o})}dk\right)
\left\vert \left\langle \pi_{\nu}(g)f_{i},u\right\rangle \right\vert ^{2}.
\]
Since $\bar{U}$ is compact there exists a constant $\xi$ and $L>0$ such that
\[
\int_{K}a(kg)^{2\operatorname{Re}(\nu-\nu_{o})}dk\leq L\left\Vert g\right\Vert
^{2\xi},\nu\in\bar{U}.
\]
Thus%
\[
\left\Vert u\right\Vert _{\nu_{o},s}^{2}\leq L\sum_{i}\int_{G}\left\vert
\left\langle \pi_{\nu}(g)f_{i},u\right\rangle \right\vert ^{2}\left\Vert
g\right\Vert ^{-2s+2\xi-d_{o}}dg=L\left\Vert u\right\Vert _{\nu,s-\xi}^{2}.
\]

\end{proof}

If $(\mu,Z)$ is a holomorphic family of objects in $H(\mathfrak{g},K)$ based
on the complex manifold $X$ then a correspondence $x\longmapsto\lambda_{x}\in
Z^{\ast}$ will be called Holomorphic if $x\mapsto\lambda_{x}(v)$ is
holomorphic for all $v\in Z$. A holomorphic correspondence $x\mapsto
\lambda_{x}$ with $\lambda_{x}\in\left(  Z_{x}\right)  _{\operatorname{mod}%
}^{\ast}$ is said to be of local uniform moderate growth if for each compact
subset $\omega\subset X$ there exists $d_{\omega}$ such that if $x\in\omega$
and $v\in Z$ then%
\[
\left\vert f_{\lambda_{x},v}(g)\right\vert \leq C_{v}\left\Vert g\right\Vert
^{d_{\omega}}%
\]
for $v\in Z,x\in X,g\in G$.

\begin{theorem}
\label{holextension}Keep the notation above. Let for $\nu\rightarrow
\lambda_{\nu}\in\left(  I_{\sigma}\right)  _{\operatorname{mod}}^{\ast}$ be
holomorphic on an open subset $W\subset$ $\left(  \mathfrak{a}_{2}\right)
_{\mathbb{C}}^{\ast}$ \ and of local uniform moderate growth. Then for each
$\nu$, $\lambda_{\nu}$ extends to an element of $\left(  I_{\sigma}^{\infty
}\right)  ^{\prime}$ and the map%
\[
\left(  \mathfrak{a}_{2}\right)  _{\mathbb{C}}^{\ast}\rightarrow\left(
I_{\sigma}^{\infty}\right)  ^{\prime}%
\]
given by $\nu$ maps to the continuous extension of $\lambda_{\nu}$ is weakly holomorphic.
\end{theorem}

The proof follows the method of the proof of Proposition 11.6.2 in
\cite{RRG1-11} to prove a continuous version of the theorem. The holomorphic
version will be derived from the continuous version. Let $\nu_{o}\in\left(
\mathfrak{a}_{2}\right)  _{\mathbb{C}}^{\ast}$ $\ $and let $U\subset W$ be an
open neighborhood of $\nu_{o}$ with compact closure in $\left(  \mathfrak{a}%
_{2}\right)  _{\mathbb{C}}^{\ast}$ then as above there exists $B_{U},s_{U}$
such that%
\[
\left\Vert \pi_{\nu}(g)\right\Vert \leq B_{U}\left\Vert g\right\Vert ^{r_{U}%
}.
\]
Also since the family $\lambda_{\nu}$ is of local uniform moderate growth
there exists for each $u\in A_{u}$ such that
\[
\overset{}{(\ast)}\left\vert f_{\lambda_{\nu},u}(g)\right\vert \leq
A_{u}\left\Vert g\right\Vert ^{m}%
\]
for $x\in U$. Set $s\geq\max\left\{  r_{U},m\right\}  .$As above, let
$f_{1},...,f_{n}$ be an orthonormal basis of a $K$ and $Z(\mathfrak{g}%
)$--invariant subspace of $W$ in $H_{K}$ such that $I_{\sigma}=d\pi_{\nu
}\left(  U(\mathfrak{g})\right)  W$ for $\nu\in U$ (Theorem \ref{loc-finite}).
If $v,w\in H$ we have as above
\[
\left\langle v,w\right\rangle _{\nu,s}=\sum_{i=1}^{n}\int_{G}\left\langle
\pi_{\nu}(g)f_{i},w\right\rangle \left\langle v,\pi_{\nu}(g)f_{i}\right\rangle
\left\Vert g\right\Vert ^{-2s-d_{o}}dg.
\]
This integral converges uniformly in $x\in U$ since there exists $C>0$ such
that
\[
\left\vert \left\langle \mu_{x}(g)v_{i},w\right\rangle \left\langle v,\mu
_{x}(g)v_{i}\right\rangle \right\vert \leq L_{s}\left\Vert g\right\Vert
^{2s}\left\Vert v\right\Vert \left\Vert w\right\Vert
\]
which also implies%
\[
\left\Vert v\right\Vert _{\nu,s}^{2}\leq M_{s}\left\Vert v\right\Vert ^{2}.
\]
Set $H_{\nu,s}$ equal to the Hilbert space completion of $H$ with respect to
$\left\langle ...,...\right\rangle _{\nu,s}$ for $x\in U$. Noting that
relative to $\left\langle ...,...\right\rangle _{\nu,s}$ the action of $K$ is
unitary, so the action of $K$ on $H$ extends to $H_{\nu,s}$. Let $\hat{\pi
}_{\nu}(g)=\pi_{\nu}(g^{-1})^{\ast}$ with respect to $\left\langle
...,...\right\rangle $. If $\sigma$ is unitary and if $Lie(A_{1})\subset
\ker\rho$ then $\hat{\pi}_{\nu}=\pi_{-\bar{\nu}}$. Also note that%
\[
\left\Vert \hat{\pi}_{\nu}(h)u\right\Vert _{\nu,s}^{2}=\sum_{i=1}^{n}\int%
_{G}\left\vert \left\langle \pi_{\nu}(g)f_{i},\hat{\pi}_{\nu}(h)u\right\rangle
\right\vert ^{2}\left\Vert g\right\Vert ^{-2s-d_{o}}dg=
\]%
\[
\sum_{i=1}^{n}\int_{G}\left\vert \left\langle \pi_{\nu}(h^{-1}g)v_{i}%
,v\right\rangle \right\vert ^{2}\left\Vert g\right\Vert ^{-2s-d_{o}}%
dg=\sum_{i=1}^{n}\int_{G}\left\vert \left\langle \pi_{\nu}(g)v_{i}%
,v\right\rangle \right\vert ^{2}\left\Vert hg\right\Vert ^{-2s-d_{o}}dg.
\]
If $h,g\in G$
\[
\left\Vert g\right\Vert =\left\Vert h^{-1}hg\right\Vert \leq\left\Vert
h^{-1}\right\Vert \left\Vert hg\right\Vert =\left\Vert h\right\Vert \left\Vert
hg\right\Vert
\]
so
\[
\left\Vert hg\right\Vert \geq\left\Vert h\right\Vert ^{-1}\left\Vert
g\right\Vert .
\]
Hence $\left\Vert hg\right\Vert ^{-2s-d_{o}}\leq\left\Vert h\right\Vert
^{2s+d_{o}}\left\Vert g\right\Vert ^{-2s-d_{o}}$. Thus%
\[
\left\Vert \hat{\pi}_{\nu}(h)v\right\Vert _{1,x}^{2}\leq\left\Vert
h\right\Vert ^{2s+d_{o}}\left\Vert v\right\Vert _{1,x}^{2}.
\]
Using this, it is easily seen that for each $x\in U$, $\hat{\pi}_{\nu}(h)$
extends to a strongly continuous representation of $G$ on $H_{\nu,s}$.

I. The $K-C^{\infty}$ vectors of $(\hat{\pi}_{\nu,s},H_{\nu,s})$ are the same
as the $G-C^{\infty}$ vectors.

To prove this assertion note that if $C$ is the Casimir operator of $G$
corresponding to the invariant form $B$ (see the beginning of section 2). Also
set $C_{K}$ equal to the Casimir operator of $K$ corresponding to
$B_{|\mathfrak{k}}.$Set $\Delta=C-2C_{K}$. Then elliptic regularity implies
that the $C^{\infty}$ vectors of $G$ in $H_{1,x}$ are the completion of
$V_{K}=H_{K}=\left(  H_{1}\right)  _{K}$ with respect to the seminorms
$p_{k,x}(v)=\left\Vert d\hat{\pi}_{\nu}(h)(\Delta^{k})v\right\Vert _{1,x}$.

One has
\[
\Delta^{k}=\sum_{j=0}^{k}(-2)^{j}\binom{k}{j}C^{k-j}C_{K}^{j}.
\]
If $v\in V$ then
\[
\left\Vert d\hat{\mu}_{x}(C)v\right\Vert _{1,x}^{2}=\sum_{i=1}^{n}\int%
_{G}|\left\langle f_{i},\hat{\pi}_{\nu}(g)d\hat{\pi}_{\nu}((C)v\right\rangle
|^{2}\left\Vert g\right\Vert ^{-2s-d_{o}}dg=
\]%
\[
\sum_{i=1}^{n}\int_{G}|\left\langle d\mu_{x}(C)f_{i},\hat{\pi}_{\nu
}(h)(g)v\right\rangle |^{2}\left\Vert g\right\Vert ^{-2s-d_{o}}dg.
\]
Now $d\pi_{\nu}(C)f_{i}=\sum a_{ji}(\nu)f_{j}$ hence $\left\langle d\pi_{\nu
}(C)f_{i},\hat{\pi}_{\nu}(g)v\right\rangle =\sum a_{ji}(x)\left\langle
f_{j},\widehat{\pi}_{\nu}(g)v\right\rangle .$Hence setting $A=\max_{ij},_{x\in
U}\{|a_{ij}(x)|\}$
\[
|\left\langle d\pi_{\nu}(C)f_{i},\widehat{\pi}_{\nu}(g)v\right\rangle
|=|\sum_{j}a_{ji}\left\langle f_{j},\widehat{\pi}_{\nu}(g)v\right\rangle |\leq
A\sum_{j}\left\vert \left\langle f_{j},\widehat{\pi}_{\nu}(g)v\right\rangle
\right\vert
\]
so%
\[
\sum_{i}|\left\langle d\pi_{\nu}(C)f_{i},\widehat{\pi}_{\nu}(g)v\right\rangle
|^{2}\leq A^{2}\left(  \sum_{j}\left\vert \left\langle f_{j},\widehat{\pi
}_{\nu}(g)v\right\rangle \right\vert \right)  ^{2}\leq
\]%
\[
nA^{2}\sum_{i}|\left\langle v_{i},\widehat{\pi}_{\nu}(g)v\right\rangle |^{2}.
\]
Thus%
\[
\left\Vert d\hat{\pi}_{\nu}(C)v\right\Vert _{\nu,s}\leq\sqrt{n}A\left\Vert
v\right\Vert _{1,x}.
\]
Set $B=\sqrt{n}A.$Then%
\[
\left\Vert d\hat{\pi}_{\nu}\left(  \Delta^{k}\right)  v\right\Vert _{\nu
,s}=\left\Vert \sum_{j=0}^{k}(-2)^{j}\binom{k}{j}d\hat{\pi}_{\nu}\left(
C^{k-j}\right)  C_{K}^{j}v\right\Vert _{1,x}\leq
\]%
\[
\sum_{j=0}^{k}(2)^{j}\binom{k}{j}B^{k-j}\left\Vert C_{K}^{j}v\right\Vert
_{1,x}\leq\sum_{j=0}^{k}(2)^{j}\binom{k}{j}B^{k-j}\left\Vert (1+C_{K}%
)^{j}v\right\Vert _{1,x}%
\]
The $K-C^{\infty}$ vectors are the completion of $I_{\sigma}$ using the
seminorms $q_{k,x}(v)=\left\Vert (I+C_{K})^{k}v\right\Vert _{1,x}$. This
proves I.

Set $\mu_{1,x}(g)$ equal to the adjoint of $\hat{\mu}_{x}(g^{-1})$ with
respect to $\left\langle ...,...\right\rangle _{1,x}$. Then the space
$K$--finite vectors of $\mu_{1,x}$ is $\left(  H_{1,x}\right)  _{K}%
=H_{K}=V_{K}$ and the corresponding $(\mathfrak{g},K)$--module is the
conjugate dual to $V_{K}$ and $d\hat{\mu}_{x}$ which is the same as the action
of $d\mu_{x}$. The CW theorem implies that the space of $C^{\infty}$--vectors
of $\mu_{1,x}$ is $V$. In particular, if $u\in H_{1,x}$ then the functional
$v\longmapsto\left\langle v,u\right\rangle _{1,x}$ is a continuous functional
on $V$.

Note that if $\mu\in I_{\sigma}^{\ast}$ then for each $\gamma\in\hat{K}$ there
exists $w_{\gamma}\in I_{\sigma}(\gamma)$ such that $\mu(v)=\left\langle
v,w_{\gamma}\right\rangle $ for $v\in V(\gamma)$. Let $E\gamma$ denote the
projection of $V_{K}$ to $V(\gamma)$ corresponding to the direct sum
decomposition $V_{K}=\oplus_{\gamma\in\hat{K}}V(\gamma)$. Set $\tau_{\gamma
}(\mu)=w_{\gamma}$. $\mu_{\gamma}=\mu\circ E_{\gamma}$.

If $v\in I_{\sigma}$ and if $\gamma\in\hat{K}$ and if $\chi_{\gamma}$ is the
character of $\gamma$ and $d(\gamma)$ is the dimension of $\gamma$ and if%
\[
f_{\lambda_{\nu},u,\gamma}(g)=d(\gamma)\int_{K}\chi_{\gamma}(k)f_{\lambda
_{x},u}(k^{-1}g)dk
\]
then 1. and 2. above imply
\[
f_{\lambda_{\nu},u,\gamma}(1)=d(\gamma)\int_{K}\chi_{\gamma}(k)f_{\lambda
_{x},u}(k^{-1})dk=\lambda_{x}(E_{\gamma}u).
\]
Also 2. implies that
\[
xf_{\lambda_{x}u,\gamma}=f_{\lambda_{\nu},d\pi_{\nu}(u)u,\gamma}%
,R(k)f_{\lambda_{x},u,\gamma}=f_{\lambda_{x},\pi(k)u,\gamma},u\in
U(\mathfrak{g}),k\in K.
\]
Thus
\[
f_{\lambda_{\nu},u,\gamma}=f_{(\lambda_{\nu})_{\gamma},u}.
\]
Also
\[
f_{\left(  \lambda_{x}\right)  _{\gamma},v}(g)=\left\langle \mu_{x}%
(g)v,\tau_{\gamma}(\lambda_{x})\right\rangle .
\]
We will now show that the series $\sum\left\Vert \tau_{\gamma}(\lambda_{\nu
})\right\Vert _{\nu,s}^{2}$ converges uniformly in $x\in U$. Indeed, the Schur
orthogonality relations and the $K$ bi-invariance of the \ norm on $G$ imply
that if $v\in V_{K}$ then $(\ast)$ above implies that there exists a constant,
$C_{u}$, depending only on $u$ such that for all $x\in U$
\[
\infty>C_{u}\geq\int_{G}|f_{\lambda_{\nu},u}(g)|^{2}\left\Vert g\right\Vert
^{-2s-d_{o}}dg=\sum_{\gamma\in\widehat{K}}\int_{G}|f_{\left(  \lambda_{\nu
}\right)  _{\gamma},u}(g)|^{2}\left\Vert g\right\Vert ^{-2s-d_{o}}dg.
\]
Hence%
\[
\infty>\sum C_{f_{i}}>\sum_{i=1}^{n}\sum_{\gamma\in\widehat{K}}\int%
_{G}|f_{\left(  \lambda_{\nu}\right)  _{\gamma},f_{i}}(g)|^{2}\left\Vert
g\right\Vert ^{-2s-d_{o}}dg=
\]%
\[
\sum_{i=1}^{n}\sum_{\gamma\in\widehat{K}}\int_{G}|\left\langle \pi_{\nu
}(g)f_{i},\tau_{\gamma}(\lambda_{\nu})\right\rangle |^{2}\left\Vert
g\right\Vert ^{-2s-d_{o}}dg=
\]%
\[
\sum\left\Vert \tau_{\gamma}(\lambda_{x})\right\Vert _{\nu,s}^{2}.
\]

II. There exist constants $B_{s}$ and $k_{s}$ such that
\[
\left\Vert u\right\Vert \leq B\left\Vert (1+C_{K})^{k}u\right\Vert _{\nu
_{o},s},u\in I_{\sigma}^{\infty}.
\]
for each $s\geq s_{U}$.

To prove this assertion we note that $I_{\sigma}^{\infty}$ is the space of
$K$--$C^{\infty}$ ($=G-C^{\infty}$) vectors for $\hat{\pi}_{\sigma}$ acting on
$H$ and every $H_{\nu,s}$ with $\nu\in U$ and $s\geq s_{U}$. Indeed, this is a
poinwise theorem that was proved in \cite{RRG1-11} (3),pp. 90-91(in line 7 of
p.91 there is a misprint the power of $\left\vert t\right\vert $ the formula
on that line should be $1$ not $2$). This implies that the underlying
$(\mathfrak{g},K)$ module of each of the $(\hat{\pi},H_{\nu,s})$ is
$(d\hat{\pi}_{\nu},I_{\sigma})$. Hence the Casselmann-Wallach globalization of
$(d\hat{\pi}_{\nu},I_{\sigma})$ is $(\hat{\pi},H_{\nu,s}^{\infty})$. Sicne it
is also $(\hat{\pi},H^{\infty})$, $\left\Vert ...\right\Vert $is a continuous
norm on $H_{\nu,s}^{\infty}$ for each $s\geq s_{U}$. This completes the proof
of II.

III. Let $\xi$ be as in the previous lemma. Assume that $s-\xi\geq\max\left\{
r_{U},m\right\}  $. Then there exists $M_{U}$ and $k_{U}$ such that%
\[
\left\Vert u\right\Vert \leq M_{U}\left\Vert (1+C_{K})^{k_{U}}u\right\Vert
_{\nu,s-\xi},u\in I_{\sigma}^{\infty}.
\]

This follows from Lemma \ref{keylemma}. Which says that if $u\in I_{\sigma
}^{\infty}$%
\[
\left\Vert u\right\Vert _{\nu_{o},s}\leq L\left\Vert u\right\Vert _{\nu,s-\xi
}.
\]
Thus
\[
\left\Vert u\right\Vert \leq LM_{U}\left\Vert (1+C_{K})^{k}u\right\Vert
_{\nu,s-\xi}.
\]

Finally we have provedt

IV. The extension of $\lambda_{\nu}$ to $I_{\sigma}^{\infty}$ for $\nu\in W$
depends weakly continuously on $\nu$ (that is, if $u\in I_{\sigma}^{\infty}$
then $\nu\longmapsto\lambda_{\nu}(u)$ is continuous).

Indeed, II let $\nu_{o}\in W$ and let $U$ be as above, an open neighborhood of
$\nu_{o}$ in $W$ with compact closure then Lemma \ref{keylemma} implies that
the series
\[
(1+C_{K})^{-k_{U}}\tau_{\gamma}(\lambda_{\nu})
\]
converges in $H$ uniformly in $\nu\in U$..

We now complete the proof of the theorem. Let $\nu_{o}\in W$ and let $\nu
_{1},...,\nu_{m}$ be a basis of $\left(  \mathfrak{a}_{2}\right)
_{\mathbb{C}}^{\ast}$ such $W$ contains $\nu_{o}+\overline{D}$ with
$\overline{D}$ the \ closure of the polydisk $D=\{\sum z_{i}\nu_{i}%
||z_{i}|<1,i=1,...,m\}.$ It is enough to prove the holomorphy assertion on
$D$. If $\nu\in D,$ $\nu-\nu_{o}=\sum x_{i}(\nu)\nu_{i}$. Define $\xi_{\nu}$
for $\nu\in D$ by%
\[
\xi_{\nu}(u)=\frac{1}{\left(  2\pi i\right)  ^{m}}\int_{S^{1}\times
S^{1}\times\cdots\times S^{1}}\frac{\lambda_{\nu}(u)dz_{1}\cdots dz_{m}}%
{\prod_{j=1}^{m}z_{j}-x_{j}(\nu)}.
\]
This integral defines a holomorphic function of $\nu$ on $\nu_{o}+D$ for each
$u\in I_{\sigma}^{\infty}$. \ If $u\in I_{\sigma}$ then $\xi_{\nu}%
(u)=\lambda_{\nu}(u).$ Thus $\xi_{\nu}=\lambda_{\nu}$ on $I_{\sigma}^{\infty}$.

\section{Application to $C^{\infty}$ Eisenstein series}

This section will involve terminology that would take us too far afield to
explain completely. Also, only those that would be bored with the explanations
would be interested in the results. For details in what is omitted we suggest
Langlands \cite{EisenL}. Let $G$ be a real reductive group of inner type. Let
$\Gamma$ be a discrete subgroup of $G$ such that $\Gamma\backslash G$ has
finite volume. The results of this section will be true for the class of $G$
and $\Gamma$ described in Chapter 1 of Langlands \cite{FuncEisen}. However, we
will only consider the subclass of $G=\mathbf{G}_{\mathbb{R}}$ the real points
of an algebraic group, $\mathbf{G}$, defined over $\mathbb{Q}$ satisfying one
more condition which we will describe later in this paragraph. and $\Gamma$ a
subgroup that is of finite index in the points of a $\mathbb{Z}$--form of
$\mathbf{G}_{\mathbb{Q}}$ (the $\mathbb{Q}$--points), i.e. an arithmetic
subgroup. A cuspidal parabolic subgroup of $G$ is the normalizer $P$ of a
parabolic subgroup $\mathbf{P}$ of $\mathbf{G}$ defined over $\mathbb{Q}$.
Then $P$ has a $\mathbb{Q}$--Langlands decomposition $P=MAN$ with $N$ the
unipotent radical of $P$ and $M$ the intersection of the kernels of $\chi^{2}$
with $\chi:M_{P}\rightarrow\mathbb{R}^{\times}$ a character defined over
$\mathbb{Q}$ and $M_{P}$ is a Levi-factor of $P$ that is defined over
$\mathbb{Q}$. The other condition is that the \textquotedblleft$A$%
\textquotedblright\ in the Langlands decomposition of $G$ is trivial. Then
$\Gamma\cap P\subset MN$ and identifying $M$ with $MN/N$ then $\Gamma
_{M}=\left(  \Gamma\cap P\right)  /\left(  \Gamma\cap N\right)  $ is an
arithmetic subgroup of $M$.

Throughout this section $P$ will be a fixed Let $V$ be space of $C^{\infty}$
vectors of a closed, $M$--invariant, irreducible subspace of $L^{2}(\Gamma
_{M}\backslash M)$. Let $\sigma$ denote the right regular action of $M$ on $V
$. Let $K$ be a maximal compact subgroup of $G$ such that $M\cap K$ is maximal
compact in $M$. We consider the smooth representation $(\pi_{\nu}%
,I_{V}^{\infty})$ where $\nu\in\mathfrak{a}_{\mathbb{C}}^{\ast}$,
$\mathfrak{a}=Lie(A)$ and $I_{V}^{\infty}$ is the space of all $C^{\infty}$
functions from $K$ to $V$ such that $f(mk)=\sigma(m)f(k)$ for $m\in K\cap M$
and $k\in K$. If $f\in I_{V}^{\infty}$ define $f_{\nu}(nmak)=a^{\nu+\rho
}\sigma(m)f(k)$ for $n\in N,m\in M,a\in A,k\in K$ and $\rho(h)=\frac{1}%
{2}tr(ad(h)_{|Lie(N)})$ for $h\in\mathfrak{a}$. \ Then since the ambiguity in
the expression of an element $g\in G$ as $g=namk,n\in N,a\in A,m\in M,k\in K$
is in $M\cap K$. $f_{\nu}$ is a $C^{\infty}$ map of $G$ to $V$. We define
$\pi_{\nu}(g)f(k)=f_{\nu}(kg)$. Endow $I_{V}^{\infty}$ with the $C^{\infty}$
topology so $I_{V}^{\infty}$ is a Fr\'{e}chet space. Note that if we set
$\pi(\nu,g)=\pi_{\nu}(g)$ then $(\pi,I_{K}^{\infty})$ is a holomorphic family
of objects in $\mathcal{HF}(G)$ based on $\mathfrak{a}_{\mathbb{C}}^{\ast}$.

If $f\in I_{V}^{\infty}$ set $\mathbf{f}_{\nu}(nmak)=a^{\nu+\rho}f(k)(m)$ for
$n\in N,m\in M,a\in A,k\in K$. Then $\mathbf{f}_{\nu}\in C^{\infty}(\left(
\Gamma\cap P\right)  \backslash G)$. Consider%
\[
E(P,f,\nu)(g)=\sum_{\gamma\in\left(  \Gamma\cap P\right)  \backslash\Gamma
}\mathbf{f}_{\nu}(\gamma g).
\]
This series converges absolutely and uniformly in compacta in the set of all
$\nu\in\mathfrak{a}_{\mathbb{C}}^{\ast}$ with Re$\nu(\check{\alpha}%
)>\rho(\check{\alpha})$ for all roots $\alpha$ of $A$ acting on $Lie(N)$.
Langlands has shown that if $f$ is in $\left(  I_{V}^{\infty}\right)  _{K}$
then this series has a meromorphic continuation to all $\nu\in\mathfrak{a}%
_{\mathbb{C}}^{\ast}$. In this section a proof will be given that the
meromorphic continuation is true for all $f\in I_{V}^{\infty}$.

Note in the set above where the series converges $E(P,f,\nu)$ is an
automorphic function. Here, a smooth function, $\varphi,$ on $\Gamma\backslash
G$ is called an automorphic function if

1. $\varphi$ is $Z(\mathfrak{g})$--finite and

2. There exists $d$ such that if $u\in U(\mathfrak{g})$ looked upon as a left
invariant differential operator then $\left\vert u\varphi(g)\right\vert \leq
C_{u}\left\Vert g\right\Vert ^{d}$ for all $g\in G$.

Usually the condition that $\varphi$ is right $K$--finite is also included in
the definition.

\begin{lemma}
If $g\in G,f\in I_{V}^{\infty}$ then in the range of convergence
\[
E(P,\pi_{\nu}(g)f,\nu)=\pi_{\Gamma}(g)E(P,f,\nu).
\]
Here $\pi_{\Gamma}$ is the right regular representation of $G$ on
$\Gamma\backslash G$.
\end{lemma}

\begin{proof}
This follows from
\[
\mathbf{f}_{\nu}(g)=f_{\nu}(g)(e)
\]
with $e$ the identity element of $G$ hence of $M$.
\end{proof}

If $f\in\left(  I_{V}^{\infty}\right)  _{K}$ and $\nu_{o}\in\mathfrak{a}%
_{\mathbb{C}}^{\ast}$ then there exists an open neighborhood of $\nu_{o}$, $U
$, in $\mathfrak{a}_{\mathbb{C}}^{\ast}$ and $\alpha$ a non-zero holomorphic
function on $U$ such that $\nu\mapsto\alpha(\nu)E(P,f,\nu)(g)$ is holomorphic.
Let $\mathcal{S}(f,\nu_{o})$ be the set of pairs $(U,\alpha)$ with $U$ in
$\mathfrak{a}_{\mathbb{C}}^{\ast}$ and $\alpha$ a non-zero holomorphic
function on $U$ such that
\[
\nu\mapsto\alpha(\nu)E(P,f,\nu)(g)
\]
is holomorphic on $U$. If $W$ is an open subset of $\mathfrak{a}_{\mathbb{C}%
}^{\ast}$ with compact closure then there exists a finite subset $F_{W}%
\subset\hat{K}$ such that
\[
d\pi_{\nu}(U(\mathfrak{g}_{\mathbb{C}}))\left(  \sum_{\gamma\in F_{W}}%
I_{V}^{\infty}(\gamma)\right)  =\left(  I_{V}^{\infty}\right)  _{K}%
\]
for $\nu\in W$ (Theorem \ref{loc-finite}). Let $f_{1},...,f_{m}$ be a basis of
$\sum_{\gamma\in F_{W}}I_{V}^{\infty}(\gamma)$ then if $\nu_{o}\in W$ and
$(U_{i},\alpha_{i})\in\mathcal{S}(f_{i},\nu_{o})$ then if $\beta=\alpha
_{1}\cdots\alpha_{m}$, $Z=U_{1}\cap\cdots\cap U_{m}\cap W$ then $\nu
\mapsto\beta(\nu)E(P,f,\nu)(g)$ is holomorphic in $\nu\in Z$ for all $g\in G$
and $f\in\left(  I_{V}^{\infty}\right)  _{K}$.

\begin{proposition}
Let $Z$ be open in $\mathfrak{a}_{\mathbb{C}}^{\ast}$ with compact closure
such that there exists $\beta$ holomorphic on $Z$ and continuous in the
closure of $\ Z$ such that $\nu\mapsto\beta(\nu)E(P,f,\nu)(g)$ is holomorphic
on $Z$ for all $g\in G$ and $f\in\left(  I_{V}^{\infty}\right)  _{K}.$ Then
there exists $d_{Z}$ such that%
\[
\left\vert \beta(\nu)E(P,f,\nu)(g)\right\vert \leq C_{f}\left\Vert
g\right\Vert ^{d_{Z}},f\in\left(  I_{V}^{\infty}\right)  _{K},g\in G.
\]

\end{proposition}

\begin{proof}
Lemma 5.1 in \cite{FuncEisen} implies that if the constant terms of $\beta
(\nu)E(P,f,\nu)$ relative to $\mathbb{Q}$--rank one parabolic subgroups
$P_{i}$ containing $P$ have exponents $a^{\mu_{i,j}(\nu)}$ and if
\[
a^{\operatorname{Re}\mu_{ij}(\nu)}\leq C\left\Vert a\right\Vert ^{d_{ij}}%
\]
then
\[
\left\vert \beta(\nu)E(P,f,\nu)(g)\right\vert \leq C_{1}\left\Vert
g\right\Vert ^{\max_{ij}d_{ij}+1}%
\]
for $g\in G$ ($1$ added to the exponent is to dominate the logarithmic term in
Langlands' inequality). On the other hand, the main observation in
\cite{Const-term} implies that the $\mu_{ij}$ are restrictions of exponents of
the $(\mathfrak{g},K)$--module $(\pi_{\nu},\left(  I_{V}^{\infty}\right)
_{K}) $. This implies that $\left\Vert \mu_{ij}(\nu)\right\Vert $ is bounded
by the maximum of the norms of the Harish-Chandra parameters of $(\pi_{\nu
},\left(  I_{V}^{\infty}\right)  _{K})$ (here the norms are with respect to
the Hermitian extension of the inner product $-B(X,\theta Y)$ on
$\mathfrak{g}$). Thus since the closure of $Z$ is compact there exists $s$
such that $a^{\operatorname{Re}\mu_{ij}(\nu)}\leq C\left\Vert a\right\Vert
^{s}$ for $\nu\in Z$. Take $d_{Z}=s+1$.
\end{proof}

We are now ready to prove

\begin{theorem}
If $f\in I_{V}^{\infty}$ then $E(P,f,\nu)(g)$ has a meromorphic continuation
to $\mathfrak{a}_{\mathbb{C}}^{\ast}$.
\end{theorem}

\begin{proof}
Let $\nu_{o}\in\mathfrak{a}_{\mathbb{C}}^{\ast}$ and let $Z$ be an open
neighborhood of $\nu_{o}$ with compact closure such that there exists $\beta$
and $d_{Z}$ as above. If $f\in\left(  I_{V}^{\infty}\right)  _{K}$ define
$\lambda_{\nu}(f)=\beta(\nu)E(P,f,\nu)(e)$. Then $\lambda_{v}\in\left(
I_{V}^{\infty}\right)  _{K}^{\ast}$ and if we set $f_{\lambda_{v},f}%
(g)=\beta(\nu)E(P,f,\nu)(g)$ then the above lemma shows that $\nu
\rightarrow\lambda_{\nu}$ is of uniform moderate growth on $Z$ hence satisfies
the hypotheses of Theorem \ref{holextension} which implies that the extension
of $\lambda_{\nu}$ to $I_{V}^{\infty}$ is weakly holomorphic in $\nu$. Since
$\nu_{o}$ is arbitrary in $\mathfrak{a}_{\mathbb{C}}^{\ast}$ this completes
the proof.
\end{proof}

\end{document}